\documentclass[12pt,multicol,english]{amsart}
\usepackage{multicol}
\usepackage[T1]{fontenc}
\usepackage[utf8]{inputenc}
\usepackage[a4paper]{geometry}
\usepackage{amssymb}
\usepackage{mathtools} 
\usepackage{enumerate}
\usepackage{dsfont}
\usepackage{mathdots}
\makeatletter

\usepackage{latexsym}
\usepackage{color}

\def\subsubsection{\@startsection{subsubsection}{3}%
	\z@{.5\linespacing\@plus.7\linespacing}{-.5em}%
	{\normalfont\bfseries}}
\newtheorem{thm}{Theorem}[section]
\newtheorem{lemme}[thm]{Lemma}

\newtheorem{prop}[thm]{Proposition}

\newtheorem{coro}[thm]{Corollary}
\newtheorem{quest}[thm]{Question}

\newtheorem{fact}{Fact}

\newtheorem*{thm*}{Theorem}
\theoremstyle{definition}
\newtheorem{rem}[thm]{Remark}
\newtheorem{ex}[thm]{Example}
\newtheorem{defi}[thm]{Definition}

\def\D{{\mathbb{D}}}

\def\C{{\mathbb{C}}}

\def\N{{\mathbb{N}}}

\def\UU{{\mathcal{U}}}
\def\TT{{\mathcal{T}}}
\def\LL{{\mathcal{L}}}

\def\EE{{\mathcal{E}}}
\def\DD{{\mathcal{D}}}
\def\PP{{\mathcal{P}}}
\def\n{\vert\vert}
\def\dlow{\underline{d}}
\def\dup{\overline{d}}
\newcommand{\dlowg}[1]{\underline{d}_{#1}}
\newcommand{\dupg}[1]{\overline{d}_{#1}}

\numberwithin{equation}{section} 
\numberwithin{figure}{section} 
\numberwithin{table}{section} 

\newcommand{\indicatrice}[1]{\mathds{1}_{#1}}

\usepackage[bookmarksopen, naturalnames]{hyperref}

\makeatother

\usepackage{babel}

\begin{document}
	\title{Common frequent hypercyclicity}
	\author{S. Charpentier, R. Ernst, M. Mestiri, A. Mouze}
	
	\address{St\'ephane Charpentier, Institut de Math\'ematiques, UMR 7373, Aix-Marseille
	Universite, 39 rue F. Joliot Curie, 13453 Marseille Cedex 13, France}
	\email{stephane.charpentier.1@univ-amu.fr}
	
	\address{Romuald Ernst, Univ. Littoral Côte d'Opale, UR 2597, LMPA, Laboratoire de Mathématiques Pures et Appliquées Joseph Liouville, Centre Universitaire de la Mi-Voix, Maison de la Recherche
	Blaise-Pascal, 50 rue Ferdinand Buisson, BP 699, 62228 Calais, France}
	\email{ernst.r@math.cnrs.fr}
	
	\address{Monia Mestiri, Université d'Artois, UR2462, Laboratoire de Math\'ematiques de Lens (LML), F-62300 Lens, France}
	\email{monia.mestiri@umons.ac.be}
	
	\address{Augustin Mouze, \'Ecole Centrale de Lille, Univ. Lille, CNRS, UMR 8524 - Laboratoire Paul Painlev\'e, F-59000 Lille, France}
	\email{augustin.mouze@univ-lille.fr}
	
	\thanks{The authors are supported by the grant ANR-17-CE40-0021 of the French National Research Agency ANR (project Front).}
	
	\begin{abstract}We provide with criteria for a family of sequences of operators to share a frequently universal vector. These criteria are inspired by the classical Frequent Hypercyclicity Criterion and by a recent criterion due to Grivaux, Matheron and Menet where periodic points play the central role. As an application, we obtain for any operator $T$ in a specific class of operators acting on a separable Banach space, a necessary and sufficient condition on a subset $\Lambda$ of the complex plane for the family $\{\lambda T:\,\lambda \in \Lambda\}$ to have a common frequently hypercyclic vector. In passing, this allows us to exhibit frequently hypercyclic weighted shifts which do not possess common frequently hypercyclic vectors. We also provide with criteria for families of the recently introduced operators of $C$-type to share a common frequently hypercyclic vector. Further, we prove that the same problem of common $\alpha$-frequent hypercyclicity may be vacuous, where the notion of $\alpha$-frequent hypercyclicity extends that of frequent hypercyclicity replacing the natural density by more general weighted densities. Finally, it is already known that any operator satisfying the classical Frequent Universality Criterion is $\alpha$-frequently universal for any sequence $\alpha$ satisfying a suitable condition. We complement this result by showing that for any such operator, there exists a vector $x$ which is $\alpha$-frequently universal for $T$, with respect to all such densities $\alpha$.
	\end{abstract}
	
	\maketitle
	
\section{Introduction}\label{Sec1}

For two separable Fr\'echet spaces $X$ and $Y$, let us denote by $\LL(X,Y)$ the set of all continuous operators from $X$ to $Y$. If $X=Y$, we simply write $\LL(X)=\LL(X,Y)$. A sequence $\TT=(T_n)_{n\in \N}\subset \LL(X,Y)$, where $\N$ stands for the set $\{0,1,2,\ldots \}$, is said to be \emph{universal} provided there exists a vector $x\in X$ such that for any non-empty open subset $U$ of $Y$, the return set
\[
N(x,U,\TT):=\{n\in \N:\,T_n(x)\in U\}
\]
is infinite. Such a vector $x$ is also called \emph{universal} and the set of all universal vectors for $\TT$ is denoted by $\UU(\TT)$. In the particular case where the universal sequence $\TT$ is given by the iterates $(T^n)_{n\in\N}$ of a single operator $T\in \LL(X)$, the operator $T$ is said to be \emph{hypercyclic} and we denote $N(x,U,\TT)$ by $N(x,U,T)$ and $\UU(\TT)$ by $HC(T)$. In 2006, Bayart and Grivaux \cite{Baygrifrequentlyhcop} introduced the notion of \emph{frequently hypercyclic operator} which quickly became one of the most important notion in linear dynamics. An operator $T\in \LL(X)$ is said to be \emph{frequently hypercyclic} if there exists $x\in X$ such that for any non-empty open subset $U$ of $X$, the \emph{lower density} of the return set $\dlow(N(x,U,T))$ is positive, where for any $E\subset \N$,
\[
\dlow(E):=\liminf_{n \to \infty} \frac{\text{card}([0,n]\cap E)}{n+1}.
\]
Such a vector $x$ is said to be \emph{frequently hypercyclic} for $T$ and the set of such vectors is denoted by $FHC(T)$. The notion of \emph{frequent universality} for a sequence $\TT$ of operators in $\LL(X,Y)$ can be similarly defined (see e.g., \cite{BonGE1}) and we denote by $\mathcal{F}\UU(\TT)$ the set of frequently universal vectors. Such behaviors look so unusual that it might seem difficult to come by with examples. However, as we shall see further on, there exist different ways to exhibit frequently hypercyclic operators, even within classical families of operators.
One of the most efficient tools to prove that an operator is frequently hypercyclic is a criterion given by Bayart and Grivaux in 2006, which is known as the Frequent Hypercyclicity Criterion and whose proof is constructive and based on the construction of a frequently hypercyclic vector as an infinite series. An extension of this last criterion for frequent universality called the Frequent Universality Criterion has been given by Bonilla and Grosse-Erdmann in 2009 \cite{BonGE}.
For a rich source of information on these well-studied notions and more about linear dynamics, we refer to the monographs \cite{bm,gp}.

A problem which has been extensively studied during the last decades is that of \emph{common hypercyclicity}.
For a given family $(T_{\lambda})_{\lambda\in \Lambda}$ of hypercyclic operators in $\LL(X)$, it consists in studying whether this family shares \emph{common hypercyclic vectors} or not, \textit{i.e.} deciding if the set $\bigcap_{\lambda \in \Lambda}HC(T_{\lambda})$ is empty or not.
Chapter 7 of \cite{bm} and Chapter 11 of \cite{gp} are entirely devoted to this topic.
On the one hand, since $HC(T)$ is a dense $G_{\delta}$ subset of $X$ whenever it is non-empty, the Baire Category Theorem ensures that $\bigcap_{\lambda \in \Lambda}HC(T_{\lambda})$ is non-empty whenever $\Lambda$ is a countable non-empty set and each $T_{\lambda}$ is hypercyclic.
On the other hand, it is not difficult to exhibit families of hypercyclic operators with no common hypercyclic vectors (for example the family of all hypercyclic weighted shifts on $\ell^2(\N)$, see \cite[Example 7.1]{bm}).
The first positive important result on that topic was obtained by Abakumov and Gordon \cite{AbakGor} who showed that the set $\bigcap _{\lambda >1}HC(\lambda B)$ is not empty, where $B$ is the backward shift on $\ell^2(\N)$ defined by $B(x_0,x_1,x_2,\ldots)=(x_1,x_2,x_3,\ldots)$.
Later on, Costakis and Sambarino \cite{CosSam} provided with the first \emph{criterion of common hypercyclicity} that they applied to prove the residuality of the set of common hypercyclic vectors for multiples of the backward shift and of differential operators, and also for uncountable families of translation operators and of specific families of weighted shifts.
The approach used by Costakis and Sambarino, based on the Baire Category Theorem, has been developed and modified by many authors to produce new criteria and to find common hypercyclic vectors for other uncountable families of classical operators, such as adjoint of multiplication operators, or composition and convolution operators (see e.g., \cite{Baycompop, BayGriMor, BayMathIndiana, ChanSand, CosCesaro, GalPar}).
A second approach to the problem, relying on some group arguments, produced some of the most striking results. Indeed, Le\'on-M\"uller's Theorem which asserts that for any $T\in \LL(X)$ and any $\lambda \in \C$, $|\lambda|=1$, $HC(T)=HC(\lambda T)$ \cite{LeonMul} can be viewed as an extremely strong result of common hypercyclicity.
Their idea, which exploits the group structure of the torus $\mathbb{T}=\{z\in \C:\,|z|=1\}$, was extended by several authors to families of operators forming groups or semigroups, and then combined with the first approach to produce some new and strong results (e.g., \cite{BayIMRN, BesRacsam, ConMulPer, Shkarin, Tsirivas}). However, such an approach cannot be performed without an underlying group or semigroup structure.
Finally, we should mention that the non-existence of common universal vectors has also been studied (see e.g., \cite{BayIMRN, bm, CosTsiVla, gp}).

These kind of questions have also been considered for stronger notions. In particular, Mestiri \cite{Moniaphd,Monia} recently studied these questions for the intermediate notion of $\UU$-frequent hypercyclicity, introduced by Shkarin in 2009 \cite{Shkarinup}, which lies between hypercyclicity and frequent hypercyclicity. The formal definition of $\UU$-frequent hypercyclicity is the same as that of frequent hypercyclicity given above except that the lower density is replaced by the \emph{upper density} $ \dup$ where for any $E\subset \N$,
\[
\dup(E):=\limsup_{n \to \infty} \frac{\text{card}([0,n]\cap E)}{n+1}.
\]
Despite the proximity with the definition of frequent hypercyclicity, $\UU$-frequent hypercyclicity often appears to be ``closer'' to hypercyclicity. An illustration of this ambivalence is shown by the residuality of the set of all $\UU$-frequently hypercyclic vectors for $T$ whenever it is non-empty \cite[Proposition 21]{BayRuz}, like in the hypercyclic case. We shall see later, that this is a major difference with frequent hypercyclicity. This allows oneself to use Baire Category arguments while studying common $\UU$-frequent hypercyclicity. 
Up to now, criteria of existence and of non-existence of such vectors have been obtained. We refer to \cite{Moniaphd,Monia} and the references therein for a complete and up-to-date overview on the subject.

In comparison, \emph{common frequent hypercyclicity} for families of operators has been considered in only a very few amount of papers. A probable reason is that the Baire Category approach, which proved to be so effective for the previous notions, drastically fails for this one. Indeed, the set $FHC(T)$ has been proven to be always meager (\textit{i.e.}, contained in the complement of a residual set) \cite[Corollary 19]{BayRuz} which seems to indicate that even for two operators the existence of common frequently hypercyclic vectors is a rare phenomenon. Yet, despite a first glimpse into this topic in the case of two multiples of the backward shift by Bayart and Grivaux \cite{Baygrifrequentlyhcop}, to our knowledge no example of a couple of operators without common frequently hypercyclic vectors has been exhibited yet. This leads us to asking the following question.

\begin{quest}\label{Q2op}
	Is it possible to provide with an example of two operators sharing no common frequently hypercyclic vectors?
\end{quest}

Moreover, in contrast to the hypercyclic case, it turns out that the frequently hypercyclic multiples $\lambda B$, $\lambda \in \Lambda \subset (0,+\infty)$, of the backward shift on $\ell^2(\N)$ have no common frequently hypercyclic vectors as soon as $\Lambda$ is uncountable \cite[Theorem 4.5]{Baygrifrequentlyhcop}. 
This result has a straightforward extension to any $T\in \LL(X)$ instead of $B$ \cite[Proposition 6.4]{BayIMRN}. Thus the first question that arises is the following

\begin{quest}\label{Qmult}
	Can one exhibit non-trivial necessary and/or sufficient conditions to common frequent hypercyclicity for multiples of a single operator?
\end{quest}

Anyway, let us mention that the previously mentioned group approach to common hypercyclicity perfectly fits to frequent hypercyclicity.
For example, Bayart and Matheron proved that $FHC(\lambda T)=FHC(T)$ for any $\lambda \in \mathbb{T}$, obtaining a \emph{frequent hypercyclicity version} of Le\'on-M\"uller's result \cite[Theorem 6.28]{bm}. This approach has been pursued further in \cite{BayIMRN} (see also \cite{ConMulPer}) and led to several nice results of common frequent hypercyclicity for families of operators forming strongly continuous groups or semigroups (like translation operators on $H(\C^d)$ and composition operators induced by non-constant Heisenberg translations on the Hardy space of the Siegel half-space).
All in all, except when actions by strongly continuous groups or semigroups are involved, so far no general criteria for common frequent hypercyclicity are known. In particular, the Baire Category approach being ineffective, one may wonder if a constructive approach could still lead somewhere.

\begin{quest}\label{Qcrit}
	Can we give constructive criteria for common frequent hypercyclicity?
\end{quest}

\medskip

In this paper we aim to contribute to the study of common frequent hypercyclicity by exploring three directions that all take their roots in the constructive Frequent Hypercyclicity (or Universality) Criterion. Indeed, we begin by following a constructive approach to the problem of the existence of common frequently hypercyclic vectors. Our first result is a criterion of common frequent universality (Theorem \ref{thmgeneral}) which is a natural strengthening of the \emph{Frequent Universality Criterion} given in \cite{BonGE} (and of the classical Frequent Hypercyclicity Criterion \cite{Baygrifrequentlyhcop}) giving an answer to Question \ref{Qcrit}. We state it for $F$-spaces, \emph{i.e.}, for complete and metrizable topological vector spaces. Let us recall that by definition, a Fréchet space is a locally convex $F$-space (see e.g., \cite{Rudin}).

\begin{thm}\label{thmgeneral}Let $X$ be a $F$-space, $Y$ a separable $F$-space and $\mathcal{T}_i=(T_{i,n})_{n\in \N}$, $i\in\N$, be countably many sequences of continuous linear operators from $X$ to $Y$. We assume that there exist a dense subset $Y_0$ of $Y$, mappings $S_{i,n}:Y_0\to X$, $i,n\in\N$, and a real number $c>1$ such that for every $y\in Y_0$,
	
	\begin{enumerate}
		\item \label{m1} the series $\sum_{0\leq n\leq m}T_{i,m}(S_{i,m-n}(y))$ converges unconditionally, uniformly for $m\in \N$ and $i\in \N$;
		\item \label{m2} the series $\sum_{n\geq 0}T_{i,m}(S_{i,m+n}(y))$ converges unconditionally, uniformly for $m\in \N$ and $i\in \N$;
		\item \label{m3} the series $\sum_{n\geq (c-1)m}T_{i,m}(S_{j,m+n}(y))$ converges unconditionally, uniformly for $m\in \N$ and $i\neq j\in \N$;
		\item \label{m4} the series $\sum_{\frac{c-1}{c}m\leq n\leq m}T_{i,m}(S_{j,m-n}(y))$ converges unconditionally, uniformly for $m\in \N$ and $i\neq j\in \N$;
		\item \label{m5} the series $\sum_{n\geq 0}S_{i,n}(y)$ converges unconditionally, uniformly for $i\in \N$;
		\item \label{m6} the sequence $(T_{i,n}(S_{i,n}(y)))_{n\in\N}$ converges to $y$, uniformly for $i\in \N$.
	\end{enumerate}
	Then there exists a vector $x\in X$ which is frequently universal for every $\mathcal{T}_i$, $i\in \N$.
\end{thm}
In particular, one may notice that each $\mathcal{T}_i$, $i\in \N$, satisfies (1), (2), (5) and (6) if and only if it satisfies the Frequent Universality Criterion given in \cite{BonGE} .
With the help of this new criterion, we are then able to get necessary or sufficient conditions on a subset $\Lambda$ of $\C$ for the set $\bigcap_{\lambda \in \Lambda}FHC(\lambda B)$ to be non-empty, when $X$ is a Banach space and $T \in \LL(X)$ giving answers to Question \ref{Qmult}. For example, we will get the following:
\begin{thm*}Let $B$ be the backward shift on $\ell^2(\N)$ and let $\Lambda \subset \C$. The set $\bigcap_{\lambda \in \Lambda}FHC(\lambda B)$ is non-empty if and only if the set $\{|\lambda|:\,\lambda \in \Lambda\}$ is a countable relatively compact non-empty subset of $(1,+\infty)$.
\end{thm*}
Actually, this theorem is valid for more general classes of (unilateral) weighted shifts on $\ell^2(\N)$, see Corollary \ref{CorB_w}. For any operator $T \in \LL(X)$, sufficient or necessary conditions on $\Lambda$ are given, involving e.g., the spectral radius of $T$, for $\Lambda$-multiples of $T$ to share frequent hypercyclic vectors. In full generality, our sufficient condition coincides with the assumption of a criterion of \emph{common hypercyclicity} given by Bayart and Matheron \cite[Proposition 4.2]{BayMathIndiana}. Our general criterion of common frequent universality is also applied to countable families of weighted shifts, differential operators or adjoint of multiplication operators (which may not be multiples of a single operator). In passing, we produce two frequently hypercyclic weighted shifts \textit{without} common frequently hypercyclic vectors answering Question \ref{Q2op}.

\medskip

Our second direction of research is motivated by the recent exhibition of a new constructive criterion for frequent hypercyclicity, based on the periodic points of the operator, given by Grivaux, Matheron and Menet \cite{GMM}. The assumptions of this criterion turn out to be more general than that of the classical Frequent Hypercyclicity Criterion, but much less easy to check with usual operators
On the other hand, Menet introduced a new class of operators, the so-called operators of $C$-type \cite{MenetTransactions}, conceived as a very rich source of counter-examples to several difficult problems, to which their new criterion for frequent hypercyclicity is very well adapted. This new criterion being more general, giving a criterion for common frequent hypercyclicity also based on the periodic points appears then as a natural task to supplement our answer to Question \ref{Qcrit}. 
Such a  criterion is given by Theorem \ref{thmFHCGMM}. Then, we illustrate how it can be applied to classes of operators of $C$-type.

\medskip

Our last direction of research is based on a recent result of Ernst and Mouze who proved \cite{ErMo,ErMo2} that any operator satisfying the usual Frequent Universality Criterion in fact enjoys a stronger form of frequent universality related to \emph{generalized lower densities} $\underline{d}_{\alpha}$ indexed by sequences of non-negative real numbers $\alpha$ satisfying suitable conditions, and first studied by Freedman and Sember \cite{Freedman}. They generalize the usual lower density in the sense that the latter coincides with any generalized lower density associated to a constant sequence $(a,a,a,\ldots)$, for any $a>0$. Moreover, if $\alpha \lesssim \beta$ (meaning $\alpha_k/\beta_k$ is eventually non-increasing), then $\dlowg{\beta}(E)\leq\dlowg{\alpha}(E)$ for every $E\subset \N$. The relation $\lesssim$ thus allows oneself to define scales of generalized lower densities.
So, every admissible $\alpha$ gives rise to a notion of $\alpha$-frequent universality generalizing the classical frequent hypercyclicity and, roughly speaking, the faster the sequence $\alpha$ grows, the stronger the associated notion of $\alpha$-frequent hypercyclicity is.
On the one hand, in \cite{ErMo, ErMo2} the authors proved that no operator can be $\alpha$-frequently hypercyclic for 
$\alpha=(e^k)_{k\geq 1}$ or growing faster. On the other hand, one of their main results states that any operator $T\in \LL(X)$ which satisfies the Frequent Universality Criterion is not only frequently hypercyclic but even  $\alpha$-frequently universal whenever there exists $s\geq 2$ such that $\alpha\lesssim \mathcal{D}^s$, where 
$\mathcal{D}^s:=(\exp(k/(\log_{(s)}(k))))_{k\geq k_0}$ for some $k_0\geq1$ depending on $s$, and $\log_{(s)}=\log \circ \log \circ \ldots \circ \log$, $\log$ appearing $s$ times. In short, this illustrates that one can approach the limiting growth as close as desired. 
In view of the topic of the paper, one may be interested in the existence of common frequent hypercyclic vectors in two different senses: the usual one, \textit{i.e.} the existence of $\alpha$-frequently hypercyclic vectors that are common to several operators; one may also think, for one given operator, of the existence of vectors which are $\alpha$-frequently hypercyclic for several sequences $\alpha$. These considerations lead us to two natural questions. For $T\in \LL(X)$ we denote by $FHC_{\alpha}(T)$ the set of all $\alpha$-frequently hypercyclic vectors for $T$.

\begin{quest}\label{Qalpha}
	Let $A$ denote the set of sequences $\alpha$ for which there exists $s\geq 2$ such that $\alpha \lesssim \mathcal{D}^s$ and let $T\in \LL(X)$.
		\begin{enumerate}
			\item Let $\Lambda \subset (0,+\infty)$ and $B\subset A$ be non-trivial. Under what condition do we have $\bigcap_{(\lambda,\beta) \in \Lambda \times B}FHC_{\beta}(\lambda T)\neq \emptyset?$\label{Qalpha1}
			\item If $T$ satisfies the Frequent Hypercyclicity Criterion, is the set $\bigcap_{\alpha \in A}FHC_{\alpha}(T)$ non-empty?\label{Qalpha2}
		\end{enumerate}
\end{quest}

We will give a positive answer to the second question (Proposition \ref{prop_common_dens}) and show that the first one has a strongly negative answer if $\Lambda$ is any non-trivial subset of $(0,+\infty)$ and even when $B$ is reduced to a single generalized density which grows faster than $(e^{\log (k)\log_{(s)}(k)})_{k\geq k_0}$ for some positive integer $s$ (Theorem \ref{propRomu}). To illustrate that the growth condition that appears in the previous answer is somehow optimal, we should mention that $FHC_{\beta}(T)=FHC(T)$ whenever $\beta$ has a growth at most polynomial (\textit{i.e.}, $\beta \lesssim (k^r)_{k\geq 1}$ for some $r> -1$) \cite[Lemma 2.10]{ErMo}. Moreover, combined with our first common frequent hypercyclicity criterion, this also gives a positive answer to Question \ref{Qalpha} (\ref{Qalpha1}) for some non-trivial $\Lambda$ and the set $B$ of sequences with at most polynomial growth.

\medskip

The paper is organized as follows. Section \ref{Sec2} is devoted to our first general criterion of common frequent universality and to some developments in various directions. In Section \ref{sec-per}, we focus on our second criterion for common frequent hypercyclicity involving periodic points. Then,  in Section \ref{Sec-alpha}, we turn ot the study of common $\alpha$-frequent hypercyclicity in both senses detailed above. Finally, we conclude the paper by a brief evocation of a possible exploration of common frequent hypercyclicity from an ergodic point of view in Section \ref{Sec-ergodic}.

\section{Common frequent universality for countable families of operators}\label{Sec2}

\subsection{A general criterion}

The main purpose of this section is to prove Theorem~\ref{thmgeneral} and to derive some corollaries. Since we will be working with $F$-spaces in this section, the notation $\Vert \cdot \Vert$ will stand for any $F$-norm defining the topology of the $F$-space. 

We first recall the definition of uniform unconditional convergence that is needed to fully understand the hypotheses of Theorem \ref{thmgeneral}.

\begin{defi} Let $\Lambda$ be a set. We say that the series $\sum_{n\in\N} x_{\lambda ,n}$, $\lambda \in \Lambda$ in $X$ converges unconditionally uniformly for $\lambda \in \Lambda$ if, for every $\varepsilon>0$, there is some $N\in \N$ such that for any finite set $F\subset \{N,N+1,\ldots\}$, one has
	\[
	\left\Vert \sum_{n\in F}x_{\lambda ,n}\right\Vert<\varepsilon
	\]
for every $\lambda \in \Lambda$.
\end{defi}

For the proof of Theorem \ref{thmgeneral}, we will make use of the following refinement of \cite[Lemma 6.19]{bm} and of ideas developed in \cite{BayRuz}.

\begin{lemme}\label{lemme-refi}For every $K>1$ and every countable family $(N_p(i))_{p\in \N}$, $i\in \N$, of increasing sequences of positive integers, there exists a countable family $(E_p(i))_{p\in \N}$, $i\in \N$, of sequences of subsets of $\N$ with positive lower density, such that for every $(p,i),(q,j)\in \N^2$ and every $(n,m)\in E_p(i)\times E_q(j)$,
	\begin{enumerate}
		\item $\min(E_p(i))\geq N_p(i)$;\label{Lemme-refi1}
		\item if $n\neq m$, then $|n-m|\ge \max(N_p(i),N_q(j))$;\label{Lemme-refi2}
		\item if $(p,i)\neq (q,j)$ and $n>m$, then $n\ge Km$.\label{Lemme-refi3}
	\end{enumerate}
\end{lemme}

\begin{proof}
	Let $K>1$ and for every $i\in \N$, let $(N_p(i))_{p\in \N}$ be an increasing sequence of positive integers. For every $i\in\N$, let us denote by $(A_p(i))_{p\in \N}$ a sequence of pairwise disjoint subsets of $\N$ with bounded gaps.
	We fix two real numbers $0<\varepsilon<1/2$ and $a>1$ satisfying
	\begin{equation}\label{eqcondia}
	\frac{1-2\varepsilon}{1+2\varepsilon}a>K.
	\end{equation}
	For $0<\eta<1$, let us set $I_{u}^{\eta}=[(1-\eta)a^u,(1+\eta)a^u]$, $u\in\N$. Then, for every $(p, i) \in \N^2$, we define  
	\[
	E_p(i)=\bigcup_{u\in A_p(i)}\left(I_{u}^{\varepsilon}\cap (N_{p}(i)\N)\right).
	\]
	By definition, \eqref{Lemme-refi1} clearly holds and (\ref{Lemme-refi2}) is satisfied whenever $(p,i)=(q,j)$. To prove that \eqref{Lemme-refi2} also holds for any $(p,i)\neq (q,j)$, we first remark that for every $u\in A_p(i)$, the inclusion
	\[
	I_{u}^{\varepsilon}+[-N_p(i),N_p(i)]\subseteq I_{u}^{2\varepsilon}
	\]
	is equivalent to the inequality $N_p(i)\leq \varepsilon a^u$. From now on we assume, up to removing finitely many elements from each set $A_p(i)$, that the previous inclusion holds for any $(p,i)\in \N^2$ and any $u\in A_p(i)$.
	
	Now, we observe that $I_{u}^{2\varepsilon}\cap I_{v}^{2\varepsilon}=\emptyset$ for every $u\neq v$. Indeed, one may check that if $u>v$, $I_{u}^{2\varepsilon}$ and $I_{v}^{2\varepsilon}$ are disjoint if and only if
	\[
	\frac{1-2\varepsilon}{1+2\varepsilon}a^{u-v}>1
	\]
	which holds by \eqref{eqcondia}. Altogether we deduce that the sets $E_p(i)$, $p,i\in \N$, are disjoint and that (\ref{Lemme-refi2}) is satisfied.
	
	To check that (\ref{Lemme-refi3}) also holds, observe first that \eqref{eqcondia} implies that $K(1+\varepsilon)<(1-\varepsilon)a$. Let $(p, i), (q, j) \in \N^2$, $n\in E_p(i)$ and $m\in E_q(j)$ such that $n > m$. Then, there exist $u\in A_p(i)$ and $v\in A_q(j)$ with $u>v$ so that:
	\[
	\begin{cases}
	(1-\varepsilon)a^u\leq n\leq (1+\varepsilon)a^u\\
	(1-\varepsilon)a^v\leq m\leq (1+\varepsilon)a^v.
	\end{cases}
	\]
	Thus, we have
	\[
	Km\leq K(1+\varepsilon)a^v<(1-\varepsilon)a^{v+1}\leq (1-\varepsilon)a^u\leq n.
	\]
	This proves (\ref{Lemme-refi3}).
	
	Finally, it remains to prove that each set $E_p(i)$ has positive lower density. Let $p,i\in\N$ and $(u_k)_{k\in\N}$ be an enumeration of the set $A_p(i)$ and let $M$ be the maximal size of a gap in $A_p(i)$. Then,
	\begin{align*}
	\underline{d}(E_p(i))& \geq\liminf_{k\to\infty}\frac{\text{card}(E_p(i)\cap [0,\lceil(1+\varepsilon)a^{u_k}\rceil])}{\lceil (1+\varepsilon)a^{u_{k+1}}\rceil}\\
	&\geq \liminf_{k\to\infty}\left(\frac{2\varepsilon a^{u_k}}{2N_p(i)}-2\right)\frac{1}{a^{u_{k+1}} + 1}\\
	&\geq \liminf_{k\to\infty}\left(\frac{\varepsilon a^{u_k}}{N_p(i)}-2\right)\frac{1}{a^{u_{k}+M} + 1}\\
	&=\frac{\varepsilon}{N_p(i)a^M + 1}>0 .
	\end{align*}
	This ends the proof of the lemma.
\end{proof}

We are now ready to prove Theorem \ref{thmgeneral}.

\begin{proof}[Proof of Theorem \ref{thmgeneral}] Within the proof, the notation $\Vert \cdot \Vert$ will be indifferently used to denote an $F$-norm defining the topologies of $X$ or $Y$. Since $Y$ is separable, we can assume that $Y_0=\{y_0,y_1,\ldots\}$. Let $(\varepsilon_p)_{p\in \N}$ be a decreasing sequence of positive real numbers such that $\sum _{p\geq0} \varepsilon _p <1$ and $p\varepsilon _p\to 0$ as $p\to \infty$. We also fix an increasing sequence $(J_p)_{p\in\N}$ such that $\sum_{i\geq J_p}\varepsilon_i <\varepsilon_p$. The assumptions of the theorem imply the existence of a sequence $(N_p(i))_{i,p\in \N}$, increasing with respect to $p\in\N$ such that for every $i,p\in \N$, every finite set $F\subset \{N_p(i), N_p(i) +1,\ldots\}$, every $m\in \N$, every $q\in\{0,\ldots,p\}$, every $k\in \N$, every $l\neq k\in \N$, and every $N\geq N_p(i)$,
\begin{multicols}{2}
	\begin{enumerate}[(i)]
		\item \label{eq1} $\displaystyle{\big\Vert \sum_{\substack{n\in F\\n< m}}T_{k,m}(S_{k,m-n}(y_q))\big\Vert < \varepsilon_p}$;
		\item \label{eq2} $\displaystyle{\big\Vert \sum_{n\in F}T_{k,m}(S_{k,m+n}(y_q))\big\Vert < \varepsilon_p}$;
		\item \label{eq5} $\displaystyle{\big\Vert \sum_{\substack{n\in F\\n\geq (c-1)m}}T_{k,m}(S_{l,m+n}(y_q))\big\Vert < \varepsilon_p\varepsilon_i}$;
		\item \label{eq5bis} $\displaystyle{\big\Vert \sum_{\substack{n\in F\\n\geq (c-1)m}}T_{k,m}(S_{l,m+n}(y_q))\big\Vert < \varepsilon_{J_p}\varepsilon_p}$;
		\item \label{eq4} $\displaystyle{\big\Vert \sum_{\substack{n\in F\\\frac{c-1}{c}m\leq n\leq m}}T_{k,m}(S_{l,m-n}(y_q))\big\Vert < \varepsilon_p\varepsilon_i}$;
		\item \label{eq4bis} $\displaystyle{\big\Vert \sum_{\substack{n\in F\\\frac{c-1}{c}m\leq n\leq m}}T_{k,m}(S_{l,m-n}(y_q))\big\Vert < \varepsilon_{J_p}\varepsilon_p}$;
		\item \label{eq3} $\displaystyle{\big\Vert \sum_{n\in F}S_{k,n}(y_q)\big\Vert < \varepsilon_p\varepsilon_i}$;
		\item \label{eq6} $\big\Vert T_{k,N}(S_{k,N}(y_p)) - y_p \big\Vert < \varepsilon_p$.
	\end{enumerate}
\end{multicols}
Let $(E_p(i))_{i,p\in \N}$ be a sequence of sets given by Lemma \ref{lemme-refi} applied to the sequence $(N_p(i))_{i,p\in\N}$ and to $K=c$. We put
\[
x=\sum_{i\in \N}\sum_{p\in \N}\sum_{n\in E_{p}(i)}S_{i,n}(y_p).
\]
One easily checks that $x \in X$. Indeed, since for every $p,i\in\N$,  $\min(E_{p}(i))\geq N_p(i)$, \eqref{eq3} gives
\[
\sum_{i\in \N}\sum_{p\in \N}\bigg\Vert \sum_{n\in E_{p}(i)}S_{i,n}(y_p)\bigg\Vert < \sum_{i\in \N}\sum_{p\in \N}\varepsilon_p\varepsilon_i < \infty.
\]
Note that $x$ is even unconditionally convergent. Our goal is now to prove that $x$ is a frequently universal vector for each sequence $(T_{i,n})_{n \in\N}$, $i\in \N$. We fix $j \in \N$. Let $(r_q)_{q\in\N}$ be a sequence of positive real numbers with $r_q \to 0$ as $q\to \infty$, to be chosen later. Since the sets $E_p(i)$, $i,p\in \N$, have positive lower density, it is sufficient to prove that
\begin{equation}\label{eqbut}
\Vert T_{j,m}(x) - y_q\Vert <r_q \text{ for every }q\in\N \text{ and every }m\in E_q(j).
\end{equation}

Let us then fix $q\in \N$ and $m\in E_q(j)$. Using that $E_p(i)\cap E_q(j) =\emptyset$ if $(i,p)\neq (j,q)$ and that $x$ is unconditionally convergent in $X$, we can decompose $T_{j,m}(x)$ as follows:
\[
T_{j,m}(x)=T_{j,m}(S_{j,m}(y_q))+\overbrace{\sum_{p\in \N}\sum_{\substack{n\in E_p(j)\\n\neq m}}T_{j,m}(S_{j,n}(y_p))}^{A_m}+\overbrace{\sum_{\substack{i\in \N\\i\neq j}}\sum_{p\in \N}\sum_{\substack{n\in E_p(i)\\n\neq m}}T_{j,m}(S_{i,n}(y_p))}^{B_{m}}.
\]
First, since $m\geq N_q(j)$, \eqref{eq6} gives
\begin{equation}\label{term1}\left\Vert T_{j,m}(S_{j,m}(y_q))-y_q\right\Vert < \varepsilon_q.
\end{equation}

We next estimate $\Vert A_m \Vert$:
\[
\Vert A_m \Vert \leq \sum_{p\in\N}\left(\bigg\Vert\sum_{\substack{n\in E_p(j)\\n<m}} T_{j,m}(S_{j,m-(m-n)}(y_p))\bigg\Vert + \bigg\Vert\sum_{\substack{n\in E_p(j)\\n> m}} T_{j,m}(S_{j,m+(n-m)}(y_p))\bigg\Vert\right).
\]
Given that $|n-m|\geq \max (N_p(j),N_q(j))$ for any $n\in E_p(j)$, $n\neq m$, \eqref{eq1} and \eqref{eq2} yield
\begin{equation}\label{term2}\Vert A_m \Vert < 2\sum_{p < q}\varepsilon _q +2 \sum_{p \geq q}\varepsilon _p =: r_{1,q}.
\end{equation}

We now turn to estimating $\Vert B_{m} \Vert$. Again, by unconditional convergence of the series, we have
\[
\Vert B_{m} \Vert \leq \overbrace{\sum_{p\in\N}\sum_{\substack{i\in \N\\i\neq j}}\bigg\Vert\sum_{\substack{n\in E_p(i)\\n<m}} T_{j,m}(S_{i,m-(m-n)}(y_p))\bigg\Vert}^{B_m^1} + \overbrace{\sum_{p\in\N}\sum_{\substack{i\in \N\\i\neq j}}\bigg\Vert\sum_{\substack{n\in E_p(i)\\n> m}} T_{j,m}(S_{i,m+(n-m)}(y_p))\bigg\Vert}^{B_m^2}.
\]
We deal first with $B_m^2$. We have
\begin{align}\label{Bm20}
B_m^2 & \leq \sum_{p \geq q}\sum_{\substack{i\in \N\\i\neq j}}\bigg\Vert\sum_{\substack{n\in E_p(i)\\n> m}} T_{j,m}(S_{i,m+(n-m)}(y_p))\bigg\Vert \nonumber\\
& + \sum _{p < q}\left(\sum_{\substack{i < J_q\\i\neq j}}\bigg\Vert\sum_{\substack{n\in E_p(i)\\n> m}} T_{j,m}(S_{i,m+(n-m)}(y_p))\bigg\Vert+\sum_{\substack{i \geq J_q\\i\neq j}}\bigg\Vert\sum_{\substack{n\in E_p(i)\\n> m}} T_{j,m}(S_{i,m+(n-m)}(y_p))\bigg\Vert\right).
\end{align}
We recall that Lemma \ref{lemme-refi} was applied to $K=c$. So, for $n\in E_p(i)$ with $(i,p)\neq (j,q)$, we have $|n-m|\geq \max(N_p(i),N_q(j))$. Moreover, $n>m$ implies $n\geq cm$, hence $n-m \geq (c-1)m$. In particular, $n-m\geq \max(N_p(i),(c-1)m)$. It follows from \eqref{eq5} that
\begin{equation}\label{Bm2term1}
\sum_{p \geq q}\sum_{\substack{i\in \N\\i\neq j}}\bigg\Vert\sum_{\substack{n\in E_p(i)\\n> m}} T_{j,m}(S_{i,m+(n-m)}(y_p))\bigg\Vert \leq \sum_{p \geq q}\sum_{i\in\N} \varepsilon_i\varepsilon_p\leq \sum_{p \geq q}\varepsilon_p
\end{equation}
and that
\begin{equation}\label{Bm2term22}
\sum_{p < q}\sum_{\substack{i \geq J_q\\i\neq j}}\bigg\Vert\sum_{\substack{n\in E_p(i)\\n> m}} T_{j,m}(S_{i,m+(n-m)}(y_p))\bigg\Vert \leq \sum_{p < q}\varepsilon_p\sum_{i \geq J_q} \varepsilon_i\leq q\varepsilon_q.
\end{equation}
In the last inequality, we use that $0 < \varepsilon_q < 1$ and the fact that $\sum_{i\geq J_q} \varepsilon_i \leq \varepsilon_q$. Now, using that $n-m\geq N_q(j)$ for any $n\in E_p(i)$ with $(i,p)\neq (j,q)$, we get from \eqref{eq5bis} that
\begin{equation}\label{Bm2term21}
\sum_{p < q}\sum_{\substack{i < J_q\\i\neq j}}\bigg\Vert\sum_{\substack{n\in E_p(i)\\n> m}} T_{j,m}(S_{i,m+(n-m)}(y_p))\bigg\Vert 
\leq \sum_{p < q}\sum_{i < J_q} \varepsilon_{J_q}\varepsilon_q\leq q\varepsilon_qJ_q\varepsilon_{J_q}.
\end{equation}
Thus, \eqref{Bm20}, \eqref{Bm2term1}, \eqref{Bm2term22} and \eqref{Bm2term21} altogether give
\begin{equation}\label{Bm2} 
B_m^2 \leq \sum_{p \geq q}\varepsilon_p + q\varepsilon_q + q\varepsilon_qJ_q\varepsilon_{J_q}=: r_{2,q}.
\end{equation}

To finish, we consider $B_m^1$. We have
\begin{align}\label{Bm10}
B_m^1 & \leq \sum_{p \geq q}\sum_{\substack{i\in \N\\i\neq j}}\bigg\Vert\sum_{\substack{n\in E_p(i)\\n< m}} T_{j,m}(S_{i,m-(m-n)}(y_p))\bigg\Vert \nonumber\\
& + \sum _{p < q}\left(\sum_{\substack{i < J_q\\i\neq j}}\bigg\Vert\sum_{\substack{n\in E_p(i)\\n< m}} T_{j,m}(S_{i,m-(m-n)}(y_p))\bigg\Vert+\sum_{\substack{i \geq J_q\\i\neq j}}\bigg\Vert\sum_{\substack{n\in E_p(i)\\n< m}} T_{j,m}(S_{i,m-(m-n)}(y_p))\bigg\Vert\right).
\end{align}
For $n\in E_p(i)$ with $(i,p)\neq (j,q)$, we have $|n-m|\geq \max(N_p(i),N_q(j))$. Moreover, $n<m$ gives $n\leq m/c$, hence $\frac{c-1}{c}m \leq m-n \leq m$. So \eqref{eq4} implies
\begin{equation}\label{Bm1term1}
\sum_{p \geq q}\sum_{\substack{i\in \N\\i\neq j}}\bigg\Vert\sum_{\substack{n\in E_p(i)\\n< m}} T_{j,m}(S_{i,m-(m-n)}(y_p))\bigg\Vert 
\leq \sum_{p \geq q}\sum_{i\in\N} \varepsilon_i\varepsilon_p
\leq \sum_{p \geq q}\varepsilon_p
\end{equation}
and
\begin{equation}\label{Bm1term22}
\sum _{p < q}\sum_{\substack{i \geq J_q\\i\neq j}}\bigg\Vert\sum_{\substack{n\in E_p(i)\\n< m}} T_{j,m}(S_{i,m-(m-n)}(y_p))\bigg\Vert 
\leq \sum_{p < q}\varepsilon_p\sum_{i \geq J_q} \varepsilon_i\leq q\varepsilon_q.
\end{equation}
Now, since $m - n \geq N_q(j)$, for $n \in E_p(i)$, $(p, i) \in \N^2$, \eqref{eq4bis} yields
\begin{equation}\label{Bm1term21}
\sum_{p < q}\sum_{\substack{i < J_q\\i\neq j}}\bigg\Vert\sum_{\substack{n\in E_p(i)\\n< m}} T_{j,m}(S_{i,m-(m-n)}(y_p))\bigg\Vert 
\leq \sum_{p < q}\sum_{i < J_q} \varepsilon_{J_q}\varepsilon_q
\leq q\varepsilon_qJ_q\varepsilon_{J_q}.
\end{equation}
Thus, \eqref{Bm10}, \eqref{Bm1term1}, \eqref{Bm1term22} and \eqref{Bm1term21} imply $B_m^1 \leq r_{2,q}$ (see \eqref{Bm2} for the definition of $r_{2,q}$).

The previous inequality, together with \eqref{term1}, \eqref{term2} and \eqref{Bm2} give \eqref{eqbut}, setting $r_q=\varepsilon_q + r_{1,q} + 2r_{2,q}$ which, by assumption, tends to $0$ as $q\to +\infty$.
\end{proof}

It will often happen that in the assumptions of Theorem \ref{thmgeneral}, $S_i$ are self-mappings of $X_0$ and right inverses of the operators $T_i$ on $X_0$. It is in particular the case if $T$ satisfies the so-called Frequent Hypercyclicity Criterion. Because we will refer to it several times in the paper, we recall its statement below.

\begin{thm}[Frequent Hypercyclicity Criterion (see Theorem 6.18 in \cite{bm})]\label{FHCcriterion}Let $X$ be a separable $F$-space and $T$ a continuous linear operator on $X$. We assume that there exist a dense subset $X_0$ of $X$ and a mapping $S:X_0\to X_0$ such that for every $x\in X_0$,
	\begin{enumerate}
		\item \label{FHCm11} the series $\sum_{n\geq 0}T^n(x)$ and $\sum_{n\geq 0}S^n(x)$ converge unconditionally;
		\item \label{FHCm21} the equality $T^n(S^n(x))=x$ holds.
	\end{enumerate}
	Then $T$ is frequently hypercyclic.
\end{thm}

In this context, Theorem \ref{thmgeneral} reads as follows.

\begin{coro}\label{CFHC}Let $X$ be a separable $F$-space and let $(T_{i})_{i\in \N}$ be countably many continuous linear operators on $X$. We assume that there exist a dense subset $X_0$ of $X$, mappings $S_{i}:X_0\to X_0$, $i\in\N$, and a real number $c>1$ such that for every $x\in X_0$,
	
	\begin{enumerate}
		\item \label{m11} the series $\sum_{n\geq 0}T_i^n(x)$ and $\sum_{n\geq 0}S_i^n(x)$ converge unconditionally, uniformly for $i\in \N$;
		\item \label{m31} the series $\sum_{n\geq (c-1)m}T_i^m(S_j^{m+n}(x))$ converges unconditionally, uniformly for $m\in \N$ and $i\neq j\in \N$;
		\item \label{m41} the series $\sum_{\frac{c-1}{c}m\leq n\leq m}T_i^m(S_j^{m-n}(x))$ converges unconditionally, uniformly for $m\in \N$ and $i\neq j\in \N$;
		\item \label{m21} the sequence $T_i(S_i(x))=x$ for every $i\in \N$.
	\end{enumerate}
	Then there exists a common frequently hypercyclic vector for the family $(T_{i})_{i\in \N}$.
\end{coro}

Note that Corollary \ref{CFHC} coincides with the Frequent Hypercyclicity Criterion when the family $(T_{i})_{i\in \N}$ is reduced to a single operator. Moreover, observe that the second part of (1) is a consequence of (2) by taking $m=0$.

\medskip

These two results apply in many situations and are sometimes sharp. This is described in the next paragraphs.

\subsection{Common frequent hypercyclicity for multiples of a single operator}\label{subsec-multiple}

Let us first give necessary conditions for the existence of common frequently hypercyclic vectors for multiples of a given operator.
\subsubsection{Necessary conditions}\label{necessaryfirst}

In this paragraph, we assume that $X$ is a Banach space. We recall that if $T$ is a bounded linear operator on $X$ and $\Lambda \subset (0,+\infty)$, then for the family $(\lambda T)_{\lambda \in \Lambda}$ to have a common frequently hypercyclic vector, $\Lambda$ has to be countable \cite[Proposition 6.4]{BayIMRN}. The following proposition shows that $\Lambda$ must also satisfy two other non-trivial conditions. We will denote by $r(T)$ the spectral radius of $T$ and we recall the spectral radius formula \cite{Rudin}:
\[
r(T)=\underset{n \geq 1}{\inf}\Vert T^n\Vert ^{1/n}=\underset{n \to \infty}{\lim}\Vert T^n\Vert ^{1/n}.
\]

\begin{prop}\label{optthmmult}Let $X$ be a Banach space, $T$ a bounded linear operator on $X$ and $\Lambda \subset (0,+\infty)$ be a set with at least two elements. 
If $\Lambda$ is unbounded or $1/r(T) \geq \inf(\Lambda)$, then
	\[
	\bigcap_{\lambda \in \Lambda}FHC(\lambda T)=\emptyset.
	\]
\end{prop}

\begin{proof}First of all, let us observe that if $\lambda < 1/r(T)$, then $\lambda T$ is not hypercyclic. We only prove the case where $1/r(T)=\inf(\Lambda)$, the case $\Lambda$ unbounded being treated very similarly. Let us first assume that $1/r(T)$ is an accumulation point of $\Lambda$. Upon taking a subsequence, we can suppose that $\Lambda =(\lambda_k)_{k\in \N}$ is decreasing to $1/r(T)$. We may and shall also assume that there exists $x\in X$ which is hypercyclic for all $\lambda_kT$, $k\in \N$. We fix $x_0\in X \setminus \{0\}$ with $\Vert x_0 \Vert =1$ and denote by $\mathcal{N}_k$, $k\geq 0$, the sets respectively given by
	\[
	\mathcal{N}_0:=\bigg\{n\in \N:\, \Vert \lambda_0^nT^n( x)   \Vert < 1\bigg\}\quad 
	\text{and} \quad \mathcal{N}_k:=\bigg\{m\in \N:\, \Vert \lambda_k^mT^m (x)  - x_0\Vert < \frac{1}{2}\bigg\},\, k\geq 1.
	\]
	By assumption, each $\mathcal{N}_k$, $k\geq 0$, is infinite. So there exist two increasing sequences $(m_k)_{k\geq 1}$ and $(n_k)_{k \geq 1}$ such that for every $k \geq 1,\ m_k \in \mathcal{N}_k$ and
\[
 n_k =\max\{n < m_k:\, n \in \mathcal{N}_0\}.
\]
 Then, from the definition of $n_k$, $k\geq 1$, we get
	\begin{equation}\label{eqopt}
	\underline{d}(\mathcal{N}_0) 
	\leq \underset{k \to \infty}{\limsup} \frac{\text{card}(\mathcal{N}_0 \cap \{0,\ldots,m_k-1\})}{m_k}
	\leq \underset{k \to \infty}{\limsup}\frac{n_k}{m_k}.
	\end{equation}
	Now, by construction, we have for any $k\geq 1$,
	\[
	\Vert T^{n_k}(x) \Vert < \lambda_0^{-n_k}\quad 
	\text{and} 
	\quad \frac{\lambda_k^{-m_k}}{2} < \Vert T^{m_k}(x) \Vert \leq \Vert T^{m_k-n_k}\Vert \Vert T^{n_k}(x)\Vert.
	\]
	It follows for any $k \geq 1$,
	\[
	\frac{\lambda_0^{n_k}}{\lambda_k^{m_k}} < 2\Vert T^{m_k-n_k}\Vert,
	\]
	whence
	\begin{equation}\label{equtileopt}
	\left(\frac{\lambda_0}{\lambda_k}\right)^{n_k}\leq 2\lambda_k^{m_k-n_k}\Vert T^{m_k-n_k} \Vert \leq 2(\lambda_0 \Vert T \Vert )^{m_k-n_k}.
	\end{equation}
	Since $(\lambda_k)_{k\in \N}$ is decreasing and $n_k\to +\infty$, we first deduce from the last inequality that $m_k-n_k\to +\infty$. This gives $r(T)=\lim _{k \to \infty} \Vert T ^{m_k-n_k}\Vert ^{1/(m_k-n_k)}$. We also derive from \eqref{equtileopt} the following:
	\[
	\left(\frac{\lambda_0}{\lambda_k}\right)^{n_k/m_k}\leq 2^{1/m_k}\lambda_k^{1-n_k/m_k} \Vert T^{m_k-n_k} \Vert ^{1/m_k} \text{ for any } k \geq 1,
	\]
	which implies, using that $m_k\to + \infty$ and $m_k-n_k \to +\infty$,
	\begin{equation*}
	\limsup_{k \to \infty} \frac{n_k}{m_k}
	\leq \frac{1}{\ln(r(T)\lambda_0)} \limsup_{k \to \infty} \left(\ln(\lambda_k\Vert T^{m_k-n_k} \Vert ^{\frac{1}{m_k -n_k}})\right)
	= 0,
	\end{equation*}
	since, by assumption, $(\lambda_k)_{k\in \N}$ is decreasing to $1/r(T)$. This with \eqref{eqopt} shows that $x$ is not frequently hypercyclic for $\lambda_0 T$ when $1/r(T)$ is an accumulation point of $\Lambda$.
	
	Let us deal with the remaining case, \textit{i.e.}, $1/r(T)\in \Lambda$ but $1/r(T)$ is not an accumulation point of $\Lambda$. We will in fact prove the stronger fact that, if $1/r(T)$ and $\lambda$ are distinct and both in $\Lambda$, then $r(T)^{-1}T$ and $\lambda T$ share no frequently hypercyclic vectors. The proof goes along the same lines as above. Let us denote $\mu = 1/r(T)$. Let $\lambda \in \Lambda$ such that $\lambda \neq \mu$. By assumption $\lambda/\mu >1$. We may and shall assume that some $x\in X$ is hypercyclic for $\lambda T$ and $\mu T$ and we set, for some vector $x_0\in X$ with $\Vert x_0\Vert=1$,
	\[
	\mathcal{N}_\lambda:=\bigg\{n\in \N:\, \Vert \lambda^nT^n (x)\Vert < 1\bigg\}\quad 
	\text{and}
	\quad \mathcal{N}_\mu:=\bigg\{m\in \N:\, \Vert \mu^mT^m (x) - x_0\Vert < \frac{1}{2}\bigg\},\, k\geq 1.
	\]
	As above, since these sets are infinite, one can define an increasing sequence of integers $(m_k)_{k\in \N} \subset \mathcal{N}_\mu$, tending to $+\infty$, such that the sequence $(n_k)_{k\in \N}$ defined by
	\[
	n_k:=\max\{n<m_k:\, n\in \mathcal{N}_\lambda\}
	\]
	is increasing. We have $\underline{d}(\mathcal{N}_\lambda)\leq \limsup_{k \to \infty} \frac{n_k}{m_k}$ and, proceeding exactly as in the first part of the proof, $m_k-n_k \to +\infty$ and
	\[
	\left(\frac{\lambda}{\mu}\right)^{n_k} 
	\leq 2 \mu^{m_k-n_k} \Vert T^{m_k-n_k} \Vert ,\quad k\in \N.
	\]
	Therefore
	\[
	\underline{d}(\mathcal{N}_{\lambda})
	\leq \limsup_{k \to \infty} \frac{n_k}{m_k} 
	\leq \frac{1}{\ln(r(T)\lambda)}\limsup_{k \to \infty}\left(\ln(\mu\Vert T^{m_k-n_k} \Vert ^{\frac{1}{m_k -n_k}})\right) 
	= 0,
	\]
	so $x$ is not frequently hypercyclic for $\lambda T$.
\end{proof}

Let us make two remarks.

\begin{rem}\label{remarkpropopt}
	The proof of the previous proposition tells us a bit more than its statement. More precisely, we have shown that, if $\Lambda$ is unbounded or if $1/r(T)\in\overline{\Lambda}$, and if $x\in X$ is a common hypercyclic vector for all $\lambda T$, then it is not \emph{frequently} hypercyclic for any $\lambda T$, $\lambda\ne 1/r(T)$. If, e.g., $T$ is the backward shift $B$ on $\ell^2(\N)$, it is not difficult to see that the set $\bigcap_{\lambda >1} HC(\lambda B)$ is different from the set $HC(\mu B)$ for any $\mu >1$. In fact, by the previous, we have, for $\mu >1$,
\[
FHC(\mu B) \subset HC(\mu B) \setminus \bigcap _{\lambda >1}HC(\lambda B).
\]
		
Another interesting feature of Proposition \ref{optthmmult} (more precisely of the second part of its proof) is that it gives an idea of how to build two frequently hypercyclic operators having no common frequently hypercyclic vectors. This will be detailed later, see Corollary \ref{FHCmaisnoncommun}).
\end{rem}

\begin{rem}\label{remarkpropopt-separ} One can wonder whether non-trivial conditions for common frequent hypercyclicity of families of non-zero real multiples of a single operator remain true for more general families of operators. 

Recall that Bayart proved in \cite{BayIMRN} that a family of multiples of a single operator can admit a common frequently hypercyclic vector only if this family is countable. However, we already know that some uncountable families of operators may have common frequently hypercyclic vectors (e.g., translation operators or composition operators, see \cite{BayIMRN,Baygrifrequentlyhcop}). Moreover, the Le\'on-M\"uller theorem for frequent hypercyclicity (\cite[Theorem 6.28]{bm}), which asserts that $FHC(\lambda T)=FHC(T)$ for any $\lambda \in \C$, $|\lambda|=1$, shows that uncountable families of complex multiples of an operator on a complex separable Banach space may have common frequently hypercyclic vector.

Now, let us focus on the necessary conditions given by Proposition \ref{optthmmult}. Note that requiring the index set to be bounded is equivalent to imposing the family of real numbers $(\Vert \lambda T \Vert)_{\lambda \in \Lambda}$ to be bounded. Besides, the second condition is equivalent to asking to  the family $(r(\lambda T))_{\lambda \in \Lambda}$ to be bounded away from $1$. Therefore, our question can be rephrased as follow: for a given family $(T_i)_{i\in \N}$ of bounded linear operators on $X$ to share common frequently hypercyclic vectors, is it necessary for the families $(\Vert T_i \Vert)_{i\in \N}$ and $(r(T_i))_{i\in \N}$ to be respectively bounded and bounded away from $1$? The answer is no: consider the family $(T^p)_{p\geq 1}$ of the positive iterates of a frequently hypercyclic operator $T$ on $X$. By Ansari's theorem for frequent hypercyclicity, $FHC(T)=FHC(T^p)$ for any $p\geq 1$. However, since a hypercyclic operator cannot be power-bounded, the family $(\Vert T^p \Vert)_{p\geq 1}$ is not bounded. Moreover, if $r(T)=1$, then by the spectral radius formula, $r(T^p)=r(T)^p=1$ for any $p\geq 1$. Note that, if $X$ is complex, by Le\'on-M\"uller's theorem, the multiples $\lambda T$ with $|\lambda|=1$ also have spectral radii equal to $1$ and yet have common frequently hypercyclic vectors. An example of a frequently hypercyclic operator whose spectral radius is one is given in Corollary \ref{FHCmaisnoncommun}.
\end{rem}

Given any bounded linear operator $T$ on $X$ and any $\lambda > 1/r(T)$, there is no reason in general for $\lambda T$ to be frequently hypercyclic and, even if $\lambda T$ is frequently hypercyclic, it may not satisfy the Frequent Hypercyclicity Criterion. In the next paragraph we search for condition on a countable set $\Lambda \subset (0,+\infty)$ for multiples $\lambda T$ of some operator $T$ to have common frequently hypercyclic vectors.

\subsubsection{Sufficient conditions}\label{sufficientmultiple}
Let us fix a separable Fréchet space $X$ and a continuous linear operator $T$ on $X$. We introduce some quantities which will play an important role in the sequel. Given $X_0$ a dense subset of $X$ and a mapping $S:X_0\to X_0$ such that $T(S(x))=x$ for $x\in X_0$, we denote by
\begin{align*}
a_T(X_0,S) & =\inf\Big\{\lambda > 0:\, \sum_{n \geq 0} \frac{S^n}{\lambda^n}(x)\text{ converges unconditionally for all }x\in X_0\Big\}\\
& =\inf\{\lambda> 0:\, (\lambda ^{-n}S^n(x))_{n\in \N}\text{ is bounded for all }x\in X_0\}
\end{align*}
and
\begin{align*}
b_T(X_0,S) & =\sup\Big\{\lambda > 0:\, \sum_{n \geq 0} (\lambda T)^n (x)\text{ converges unconditionally for all }x\in X_0\Big\}\\
& =\sup\{\lambda > 0:\, ((\lambda T)^n(x))_{n\in \N}\text{ is bounded for all }x\in X_0\}.
\end{align*}
When $X$ is a Banach space, one may check that
\[
a_T(X_0,S)=\sup_{x\in X_0}\limsup_{n \to \infty} \Vert S^n(x) \Vert ^{1/n}\quad \text{and} \quad b_T(X_0,S)=\inf_{x\in X_0}\frac{1}{\limsup_{n \to \infty} \Vert T^n(x) \Vert ^{1/n}}.
\]
So, by the spectral radius formula, we have
\begin{equation}\label{defia}
a_T(X_0,S)\geq r(T)^{-1}\quad \text{and} \quad b_T(X_0,S)\geq r(T)^{-1}.
\end{equation}
Note that $b_T(X_0,S)$ may be infinite e.g., if $X_0=\bigcup_{n\geq 0}\ker(T^n)$ is dense in $X$. This is for example the case if $T$ is any weighted backward shift acting on a Fr\'echet space $X$ with an unconditional basis. More specifically, if $T$ is the backward shift $B$ on $\ell^2(\N)$, then $S$ can be taken as the forward shift $F$ and the first inequality in \eqref{defia} is in fact an equality. This yields $a_B(X_0,F)=1/\Vert B \Vert =1$ (see Paragraph \ref{subweightedshifts} for a focus on weighted shifts) and $b_B(X_0,F)=+\infty$.

\medskip

With these notations, a criterion of common hypercyclicity, due to Bayart and Matheron, can be rephrased as follows.

\begin{thm}[Proposition 4.2 in \cite{BayMathIndiana}]\label{BayMathcommon}Let $X$ be a separable Fréchet space and let $T$ be a continuous linear operator on $X$. We assume that there exist $X_0\subset \bigcup_{n\in \N}\ker (T^n)$ and a mapping $S:X_0\to X_0$ such that $X_0$ is dense in $X$ and $T(S(x))=x$ for all $x\in X_0$. Then $\bigcap_{\lambda > a_T(X_0,S)}HC(\lambda T)$ is a dense $G_{\delta}$ subset of $X$.
\end{thm}

Now let us observe that, by definition, for any $a_T(X_0,S)<\lambda < b_T(X_0,S)$ the family $(\lambda^n T^n)_{n\in \N}$ satisfies the \emph{Frequent Universality Criterion} \cite{BonGE}. So it is natural to wonder under which extra condition the family $\lambda T$, $a_T(X_0,S)<\lambda < b_T(X_0,S)$, has a common frequently hypercyclic vector. In virtue of the necessary conditions given in Paragraph \ref{necessaryfirst}, the following criterion is a quite natural extension of Bayart and Matheron's result.

\begin{thm}\label{thmmult}Let $X$ be a separable Fréchet space and $T$ a continuous linear operator on $X$. We assume that there exist a dense subset $X_0$ of $X$ and a mapping $S:X_0 \to X_0$ such that $T(S(x))=x$ for all $x\in X_0$. If $\Lambda$ is a countable relatively compact non-empty subset of $(a_T(X_0,S),b_T(X_0,S))$, then $\bigcap_{\lambda \in \Lambda}FHC(\lambda T)\neq \emptyset$.
\end{thm}

The proof is based on the following lemma, where it is assumed that $E\subset (a,b)$ with 
$a<b$ means $E=\emptyset$.

\begin{lemme}\label{lemmeutilecounta}With the assumptions of Theorem \ref{thmmult}, let $\Lambda$ be a relatively compact subset of $(a_T(X_0,S),b_T(X_0,S))$. Then there exists $c>1$ such that for any $x\in X_0$,
\begin{enumerate}[(i)]
		\item \label{i} the series $\sum_{n\geq 0}(\lambda T)^n(x)$ converges unconditionally, uniformly for $\lambda \in \Lambda$;
		\item \label{ii} the series $\sum _{n\geq 0}\left(\frac{S}{\lambda}\right)^n(x)$ converges unconditionally, uniformly for $\lambda \in \Lambda$;
		\item \label{iii} the series $\sum _{n\geq (c-1)m}(\frac{\lambda}{\mu})^{m}(\frac{S}{\mu})^n(x)$ converges unconditionally, uniformly for $m\in \N$ and $\lambda,\mu \in \Lambda$;
		\item \label{iiii} the series $\sum _{m\geq n\geq \frac{c-1}{c}m}(\frac{\lambda}{\mu})^{m-n}(\lambda T)^n(x)$ converges unconditionally, uniformly for $m\in \N$ and $\lambda,\mu \in \Lambda$.
\end{enumerate}
\end{lemme}

\begin{proof}Let us denote by $\Vert \cdot \Vert$ any continuous semi-norm on $X$. For notational simplicity, we shall denote $a=a_T(X_0,S)$ and $b=b_T(X_0,S)$. We only prove \eqref{ii} and \eqref{iii}.
The conditions \eqref{i} and \eqref{iiii} are respectively proved in a similar way.
Let $a<d<\inf(\Lambda)$. To get \eqref{ii}, it is enough to write, for $\lambda \in \Lambda$ and $m \in \N$, $(\frac{S}{\lambda})^n(x)=(\frac{d}{\lambda})^n(\frac{S}{d})^n(x)$, and use that $\frac{d}{\lambda}\leq \frac{d}{\inf(\Lambda)}<1$ and that $(\frac{S}{d})^n(x)$ is bounded for any $x\in X_0$ by some constant independent of $\lambda \in \Lambda$ and $n\in\N$.

To prove \eqref{iii}, let us fix $x\in X_0$. By assumption, the series $\sum_{n\geq 0}\left(\frac{S}{d}\right)^n(x)$ is unconditionally convergent. We also let $c>1$ be such that
\[
\frac{\sup(\Lambda)}{\inf(\Lambda)}\left(\frac{d}{\inf(\Lambda)}\right)^{c-1}\leq 1,
\]
and we write
\[
\sum _{n\geq (c-1)m}\left(\frac{\lambda}{\mu}\right)^{m}\left(\frac{S}{\mu}\right)^n(x)=\sum_{n\geq (c-1)m}\left(\frac{\lambda}{\mu}\left(\frac{d}{\mu}\right)^{c-1}\right)^{m}\left(\frac{d}{\mu}\right)^{n-(c-1)m}\left(\frac{S}{d}\right)^n(x).
\]
Since the quantity
\[
\left(\frac{\lambda}{\mu}\left(\frac{d}{\mu}\right)^{c-1}\right)^{m}\left(\frac{d}{\mu}\right)^{n-(c-1)m}
\]
is bounded by $1$ uniformly for $\lambda,\mu\in \Lambda$, $m\in \N$ and $n\geq (c-1)m$, we get \eqref{iii} by unconditional convergence of the series $\sum_{n\geq 0}\left(\frac{S}{d}\right)^n(x)$.
\end{proof}

Let us now prove Theorem \ref{thmmult}.

\begin{proof}[Proof of Theorem \ref{thmmult}]
It is enough to check that the sequences $((\lambda T)^n)_{n\in\N}$ and $((S/\lambda)^n)_{n\in\N}$, $\lambda \in \Lambda$, satisfy the assumptions \eqref{m1}--\eqref{m6} of Theorem \ref{thmgeneral}. \eqref{m6} is trivial, while \eqref{m1}, \eqref{m2} and \eqref{m5} are direct consequences of \eqref{i} and \eqref{ii} of Lemma \ref{lemmeutilecounta}. Now, \eqref{m3} and \eqref{m4} follow from \eqref{iii} and \eqref{iiii} of Lemma \ref{lemmeutilecounta}, after observing that for any $\lambda\neq \mu \in \Lambda$, $x\in X_0$,
\[
\sum_{n\geq (c-1)m}(\lambda T)^m\left(\frac{S}{\mu}\right)^{m+n}(x)=\sum _{n\geq (c-1)m}\left(\frac{\lambda}{\mu}\right)^{m}\left(\frac{S}{\mu}\right)^n(x)
\]
and
\[
\sum_{\frac{c-1}{c}m \leq n\leq m}(\lambda T)^m\left(\frac{S}{\mu}\right)^{m-n}(x)=\sum _{\frac{c-1}{c}m \leq n\leq m}\left(\frac{\lambda}{\mu}\right)^{m-n}(\lambda T)^n(x).
\]
\end{proof}

A slight modification of the proof of Theorem \ref{thmmult} yields to the following \emph{universal} version.
\begin{prop} Let $X$ be a separable Fréchet space, $T$ a bounded linear operator on $X$, $X_0$ a dense subset of $X$ and a mapping $S:X_0 \to X_0$ such that $T(S(x))=x$ for all $x\in X_0$. Let also $(\lambda_{i,n})_{n\geq 1}$, $i\in \N$, be a countable family of sequences in $(0,+\infty)$. We assume that
\begin{enumerate}
\item there exist $c,d \in (a_T(X_0,S),b_T(X_0,S))$ such that $\lambda_{i,n} \in (c^n,d^n)$ for any $i\in \N$, $n\geq 1$;
\item there exists $C>0$ such that $C^{-1}\lambda_{i,{n+m}} \leq \lambda_{i,n}\lambda_{i,m}\leq C\lambda_{i,{n+m}}$ for any $n,m,i\in \N$.
\end{enumerate}
Then
\[
\bigcap_{i\in \N} \mathcal{F}\mathcal{U}((\lambda_{i,n}T^n)_n)\neq \emptyset.
\]
\end{prop}

Together with the result of Paragraph \ref{necessaryfirst}, Theorem \ref{thmmult} gives a necessary and sufficient condition on a set $\Lambda \subset (0,+\infty)$ for common frequent hypercyclicity of the family $\lambda T$, $\lambda \in \Lambda$, for any $T$ in a certain subclass of operator acting on a Banach space. Recall that if $\lambda < 1/r(T)$, then $\lambda T$ is not hypercyclic.

\begin{coro}\label{corothm29opt}Let $X$ be a separable Banach space, $T$ be a bounded linear operator on $X$ and 
$\Lambda \subset (0,+\infty)$ with at least two elements. We assume that there exist a dense subset $X_0$ of $X$ and a mapping $S:X_0 \to X_0$ such that $T(S(x))=x$ for all $x\in X_0$. We also suppose that $a_T(S,X_0)=1/r(T)$ and $b_T(S,X_0)=+\infty$. Then, 
	\[
	\bigcap_{\lambda \in \Lambda}FHC(\lambda T)\neq \emptyset
	\]
if and only if $\Lambda$ is countable and relatively compact in $(1/r(T),+\infty)$.
\end{coro}

The following question arises. It will be investigated later for the class of weighted shifts, see Paragraph \ref{CFHCweightedshifts}.

\begin{quest}\label{questopt}For those operators $T$ such that $a_T(S,X_0)>1/r(T)$ for some $S$ and $X_0$ as in Corollary \ref{corothm29opt}, can one improve the necessary condition on $\Lambda \subset (0,+\infty)$, given in Proposition \ref{optthmmult}, so that the multiples $\lambda T$ share common frequently hypercyclic vectors?
\end{quest}
 
\medskip

To conclude the paragraph, let us combine the previous results with Le\'on-M\"uller's theorem and Ansari's theorem for frequent hypercyclicity, see respectively \cite[Theorem 6.28]{bm} and \cite[Theorem 4.7]{Baygrifrequentlyhcop}. We recall that they tell us that $FHC(\lambda T)=FHC(T)=FHC(T^p)$ for any $\lambda \in \C$, $|\lambda|=1$, and any positive integer $p$. These with Theorem \ref{thmmult} thus imply:

\begin{coro}\label{corosuffLMAnsari}Let $X$ be a separable complex Fréchet space, $T$ a bounded linear operator on $X$ and $\Lambda$ a non-empty subset of $\C$. We assume that there exist a dense subset $X_0$ of $X$ and a mapping $S:X_0 \to X_0$ such that $T(S(x))=x$ for all $x\in X_0$.  If the set $\{|\lambda|:\,\lambda\in\Lambda\}$ is a countable relatively compact subset of $(a_T(X_0,S),b_T(X_0,S))$, then \[\bigcap_{\lambda\in\Lambda}FHC(\lambda T)\neq \emptyset.\]
	\noindent Moreover, if the set $\{|\lambda|^{1/p}:\,p\in \N^*,\,\lambda \in \Lambda\}$ is a countable relatively compact subset of $(a_T(X_0,S),b_T(X_0,S))$, then
	\[
	\bigcap_{\lambda\in\Lambda,\,p\in \N^*}FHC(\lambda T^p)\neq \emptyset.
	\]
\end{coro}
Note that the second condition is a consequence of the first one if $a_T(X_0,S)<1$ and $b_T(X_0,S)=+\infty$ (e.g., for a large class of weighted shifts, see the next paragraph).

\begin{rem} When $X$ is a separable Banach space, if $a_T(X_0,S)=1/r(T)$ and $b_T(X_0,S)=+\infty$ and $\Lambda$ has at least two elements, then the sufficient conditions on $\Lambda \subset \C$ given in Corollary \ref{corosuffLMAnsari} are also necessary.
\end{rem}

\medskip

We shall make another remark.
	
\begin{rem}It should be noticed that the definitions of $a_T(X_0,S)$ and $b_T(X_0,S)$ depend \textit{a priori} on $X_0$ and $S$. In particular, it could happen that for some $T\in \LL(X)$, there exist couples $(X_0,S_0)$ and $(X_1,S_1)$ such that $a_T(X_1,S_1)< a_T(X_0,S_0)$. Thus it is tempting to introduce the quantities
\[
a_T:=\inf_{(X_0,S)}a_T(X_0,S)\quad \text{and} \quad b_T:=\sup_{(X_0,S)}b_T(X_0,S),
\]
where the infimum and the supremum are taken over all couples $(X_0,S)$ such that $X_0$ is dense in $X$ and $S:X_0\to X_0$ is such that $T(S(x))=x$ for every $x\in X_0$. But the conclusions of the previous results may not hold true replacing $a_T(X_0,S)$ by $a_T$ and $b_T(X_0,S)$ by $b_T$. Indeed, it might happen that for some $(X_0,S)$, $a_T(X_0,S)$ is very close to $a_T$ but $b_T(X_0,S)$ is very small compared to $b_T$.

However, if $T:=B_w$ is a frequently hypercyclic weighted shift acting on $\ell^p(\N)$, $1\leq p< +\infty$, it turns out that $a_T=a_T(c_{00}(\N),F_w)$ and $b_T=b_T(c_{00}(\N),F_w)=+\infty$, see the next paragraph for the formal definitions of $c_{00}(\N)$ and $F_w$. This is a consequence of Bayart and Ruzsa's theorem \cite[Theorem~4]{BayRuz}.
\end{rem}

\medskip

In the next section we concentrate our attention on common frequent hypercyclicity for the important class of weighted shifts.

\subsection{Common frequent hypercyclicity for weighted shifts}\label{subweightedshifts}

In this whole section, we assume that $X$ is a Fréchet space with an unconditional basis $(e_n)_{n\in \N}$. We call \textit{weight} a sequence of nonzero real numbers. Given a weight $w=(w_n)_{n\in \N}$, the \textit{weighted shift} $B_w$ is defined, for $x=\sum_{n\geq 0}x_ne_n \in X$, by
\[
B_w(x)=\sum_{n\geq 0}w_{n+1}x_{n+1}e_n.
\]
The series $\sum_{n\geq 0}w_{n+1}x_{n+1}e_n$ may not be convergent in $X$ for all $x\in X$ yet, by the Closed Graph Theorem, $B_w$ maps $X$ into itself if and only if it is continuous on $X$. In this case, it is equivalently defined by $B_w(e_n)=w_{n}e_{n-1}$, $n \geq 0$, with the convention $e_{-1}=0$.

For any weight $w$, $B_w$ admits a (formal) right inverse, that we denote $F_w$, given by
\[
F_w(x)=\sum_{n\geq 1}\frac{x_{n-1}}{w_{n}}e_n
\]
for $x=\sum_{n\geq 0}x_ne_n\in X$. The series $\sum_{n\geq 1}\frac{x_{n-1}}{w_{n}}e_n$ may not belong to $X$, but $F_w$ is well-defined from $c_{00}(\N):=\text{span}(e_n:\,n\geq 0)$ into itself and $F_w(e_n)=e_{n+1}/w_{n+1}$, $n \geq 0$. Note that the map $F_w$ is referred to as the \textit{forward shift} associated to the weight $w^{-1}:=(w_{n}^{-1})_{n\geq 0}$.

\medskip

We recall that a continuous weighted shift $B_w$ on $X$ is frequently hypercyclic whenever the series
\[
\sum_{n\geq 1}(w_1\ldots w_n)^{-1}e_n
\]
is convergent in $X$, see \cite[Corollary 9.14]{gp}.

\subsubsection{General criteria}

We first state a criterion of common frequent hypercyclicity for general families of weighted shifts, derived from Corollary \ref{CFHC}.

\begin{thm}\label{weightedshiftscounta}Let $X$ be a separable Fr\'echet space with an unconditional basis $(e_n)_{n \in \N}$ and $w(i)=(w_n(i))_{n\in\N}$, $i\in \N$, be countably many weights for which every $B_{w(i)}$, $i\in \N$, is a continuous operator on $X$. We assume that there exist a weight $\omega=(\omega_n)_{n\in\N}$, constants $M\geq 1$ and $0<\eta \leq 1$ with either $M=\eta=1$ or $M\neq 1$ and $\eta\neq 1$, and a constant $C>0$, such that for any $i\in \N$ and any $n\geq 0$, $m\geq 1$,
\begin{enumerate}[(i)]
	\item \label{lablab1}the series $\sum_{k \geq 1}(\omega_1\ldots\omega_{k})^{-1}e_{k}$ is unconditionally convergent in $X$;
	\item \label{lablab2}$\vert \omega_n\ldots \omega_{n+m} \vert \leq C\eta^m \vert w_n(i)\ldots w_{n+m}(i) \vert$;
	\item \label{lablab3}$C^{-1}M ^{-m}\leq \vert w_n(i)\ldots w_{n+m}(i) \vert \leq CM^m$.
\end{enumerate}
Then there exists a common frequently hypercyclic vector for the family $(B_{w(i)})_{i\in \N}$.
\end{thm}

\begin{proof}We consider $X_0=\text{span}(e_k:\,k\geq 0)$. Since $X_0$ is dense in $X$, up to taking $S_i:=F_{w(i)}$, $i \in \N$, we need only check that the assumptions \eqref{m11}--\eqref{m21} of Corollary \ref{CFHC} are satified for any $x=e_k$, $k\in \N$. Let us then fix $k\in \N$. Observe that \eqref{m21} is trivially satisfied. From now on, for $l<0$, we use the notations $e_l=0$ and $w_l(i)=0$, $i\in \N$. For any $i,j,l,m\in \N$, let us write
\[
B_{w(i)}^m(F_{w(j)}^l(e_k))=\frac{w_{k+l-m+1}(i)\ldots w_{k+l}(i)}{w_{k+1}(j)\ldots w_{k+l}(j)}e_{k+l-m}.
\]
Note that $B_{w(i)}^n(e_k)=0$ whenever $n > k$. This gives the first part of \eqref{m11} in Corollary \ref{CFHC}. Moreover, for every $i \in \N$,
\begin{equation}\label{1eq}
\sum_{n\geq 0}F_{w(i)}^{n}(e_k)=\sum_{n\geq 0}\frac{1}{w_{k+1}(i)\ldots w_{k+n}(i)}e_{k+n}.
\end{equation}
By assumption \eqref{lablab2}, we have $\vert w_{k+1}(i)\ldots w_{k+n}(i) \vert > \vert \omega_{k+1}\ldots \omega_{k+n} \vert$ for every $n\geq 1$ and $i \in \N$. So, by condition \eqref{lablab1} and using that $(e_n)_{n\in\N}$ is an unconditional basis, we get that the left-hand side term in \eqref{1eq} is unconditionally convergent in $X$, uniformly for $i$, hence the second part of \eqref{m11} in Corollary \ref{CFHC}.

Let us now turn to proving that \eqref{m31} in Corollary \ref{CFHC} holds. We write
\begin{align*}
 B_{w(i)}^m(F_{w(j)}^{m+n}(e_k))  & =  \frac{w_{k+n+1}(i)\ldots w_{k+n+m}(i)}{w_{k+1}(j)\ldots w_{k+n+m}(j)}e_{k+n}  \\
& =  \frac{w_{k +n+1}(i)\ldots w_{k+n+m}(i)}{w_{k+n+1}(j)\ldots w_{k+n+m}(j)} \frac{\omega_{k+1}\ldots \omega_{k+n}}{w_{k+1}(j)\ldots w_{k+n}(j)} (\omega_{k+1}\ldots \omega_{k+n})^{-1}e_{k+n}.
\end{align*}
Let us first assume that $M=\eta=1$. By \eqref{lablab1}, \eqref{lablab2} and \eqref{lablab3}, the series $\sum_{n\geq 0}B_{w(i)}^m(F_{w(j)}^{m+n}(e_k))$ is unconditionally convergent uniformly with respect to $m\geq 1$ and $i,j\in\N$. Therefore, the assumption \eqref{m31} is satisfied for every $c > 1$.

Let us now suppose that $M>1$ so that $\eta <1$. Let $\varrho\in (\eta ,1)$ and $c>1$ be such that $M^2\varrho^{c-1} \leq 1$. By the condition \eqref{lablab1} and unconditionality of $(e_n)_{n \in \N}$, the sequence $((\omega_{k+1}\ldots \omega_{k+n})^{-1}e_{k+n})_{n \in \N}$ is bounded. We denote by $\Vert \cdot \Vert $ any continuous semi-norm on $X$. Then, for some constant $K$ (depending only on $\eta$, $C$, $k$ and the constant of unconditionality of $(e_n)_{n \in \N}$) and thanks to the assumptions \eqref{lablab2} and \eqref{lablab3}, we have for any $n,i, j \in \N$ and $m \geq 1$,
\[
\Big\Vert B_{w(i)}^m(F_{w(j)}^{m+n}(e_k)) \Big\Vert \leq K M^{2m} \eta^n.
\]
Let us now write
\[
M^{2m}\eta ^n=(M^2\varrho^{c-1})^m\varrho^{n-(c-1)m}\left(\frac{\eta}{\varrho}\right)^n\leq \left(\frac{\eta}{\varrho}\right)^n
\]
for any $n\geq (c-1)m$ and $m \geq 1$. As $\varrho\in (\eta ,1)$, the series $\sum _{n\geq (c-1)m}M^{2m} \eta^n$ is absolutely convergent, uniformly for $m\geq 0$, and so the series $\sum _{n\geq (c-1)m}B_{w(i)}^m(F_{w(j)}^{m+n}(e_k))$ converges unconditionally, uniformly for $m \geq 0$ and $i,j\in\N$. This implies \eqref{m31} from Corollary \ref{CFHC}.

That \eqref{m41} in  Corollary \ref{CFHC} holds in this setting  is left to the reader.
\end{proof}

\begin{rem}\label{Rem_Th2.18}
	As a corollary of the proof, one may check that Theorem \ref{weightedshiftscounta} remains true if we suppose that there exist a weight $\omega=(\omega_n)_{n\in\N}$ and a constant $C>0$, such that for any $i,j\in \N$ and any $n\geq 0$, $m\geq 1$,
	\begin{enumerate}[(i)]
	\item \label{lablab11}the series $\sum_{k \geq 1}(\omega_1\ldots\omega_{k})^{-1}e_{k}$ is unconditionally convergent in $X$;
	\item \label{lablab22}$ \vert \omega_n\ldots \omega_{n+m}\vert \leq C \vert w_n(i)\ldots w_{n+m}(i) \vert$;
	\item \label{lablab33}$C^{-1}\leq \big\vert \frac{w_0(i)\ldots w_n(i)}{w_0(j)\ldots w_n(j)} \big\vert \leq C$.
	\end{enumerate}
In particular, if the family is composed of a finite (non-zero) number of frequently hypercyclic operators, then it suffices to check (\ref{lablab33}). Moreover, if two such operators satisfy that the product of their weights are equivalent then they share frequent hypercyclic operators.
\end{rem}

Let us give an example.

\begin{ex} Let $1 \leq p < + \infty$. For $\lambda \in (0,+\infty)$, let $B_{w(\lambda)}$ be the weighted shift on $\ell^p(\N)$, defined by $w_n(\lambda)=1+\lambda /n$, $n\geq 1$. In \cite{CosSam}, it is proven that $\bigcap_{\lambda >1}HC(B_{w(\lambda)})$ is residual. Now, one may check that the series
\[
\sum_{n\geq 1}\frac{1}{w_1(\lambda)\ldots w_n(\lambda)}e_n
\]
is unconditionally convergent in $\ell^p(\N)$ if and only if $\lambda > 1/p$ (where $(e_n)_{n\in \N}$ is the unit sequence in $\ell^p(\N)$). We can thus deduce from Theorem \ref{weightedshiftscounta}, applied with $M=\eta =1$, that for any countable relatively compact subset $\Lambda$ of $(\frac{1}{p},+\infty)$, one has
		\[
		\bigcap_{\lambda \in \Lambda} FHC(B_{w(\lambda)})\neq \emptyset.
		\]
Observe that, by Bayart and Ruzsa's theorem (see Remark \ref{remBayRuz} below), $B_{w(\lambda)}$ is not frequently hypercyclic on $\ell^p(\N)$ if $\lambda \leq 1/p$.
\end{ex}

\medskip

The main result of Paragraph \ref{subsec-multiple} can be also applied to weighted shifts. Let $X$ be a separable Fréchet space with an unconditional basis $(e_n)_{n\in \N}$ and set $c_{00}(\N)=\text{span}(e_n:\,n\geq 0)$. With the notations of Paragraph \ref{subsec-multiple}, a slight generalization of Abakumov and Gordon's theorem states that the set of common hypercyclic vectors for the multiples $\lambda B_w$ of a continuous weighted shift $B_w$ on $X$, $\lambda > a_{B_w}(c_{00}(\N),F_w)$, is $G_{\delta}$ and dense in $X$ see \cite[p. 178]{bm} or \cite{BayMathIndiana}. Note that this result can also be deduced from Theorem \ref{BayMathcommon}.

In this context, Theorem \ref{thmmult} reads as follows.

\begin{coro}\label{coromultipleshifts}Let $X$ be a separable Fréchet space with an unconditional basis $(e_n)_{n\in \N}$ and let $B_w$ be a continuous weighted shift on $X$. Then the set $\bigcap_{\lambda \in \Lambda}FHC(\lambda B_w)$ is non-empty whenever $\Lambda$ is any countable relatively compact non-empty subset of $(a_{B_w}(c_{00}(\N),F_w),+\infty)$.
\end{coro}

Similarly, Corollary \ref{corothm29opt} tells us that if in Corollary \ref{coromultipleshifts} we assume in addition that $a_{B_w}(c_{00}(\N),F_w)=1/r(B_w)$ and $\Lambda$ has at least two elements, then the condition becomes also necessary . Moreover Question \ref{questopt} makes sense, and seems to be a bit more accessible in this setting, especially when $X=\ell^p(\N)$, $1\leq p < +\infty$. This particular case will be investigated in the next paragraph.

\medskip

To finish, we shall make a remark.

\begin{rem}\label{remBayRuz}
Remark that in a lot of cases, the condition $\Lambda \subseteq \left(a_{B_w}(c_{00}(\N),F_w),+\infty\right)$ in Corollary \ref{coromultipleshifts} can be replaced by asking that $\lambda B_w$ is frequently hypercyclic for every $\lambda\in\Lambda$ (which is clearly the weakest thing to ask to get common frequent hypercyclicity). Indeed  for $\ell^p(\N)$ spaces with $1\leq p <\infty$, this follows from a result of Bayart and Ruzsa \cite{BayRuz} in 2015 stating that weighted shifts are frequently hypercyclic if and only if they satisfy the Frequent Hypercyclicity Criterion. This result was then extended to more general classes of spaces in \cite{CGEM}, including for example the space $H(\D)$.
\end{rem}

\subsubsection{Common frequent hypercyclicity for multiples of weighted shifts on $\ell^p(\N)$}\label{CFHCweightedshifts}
In this paragraph we specify the study led in Section \ref{subsec-multiple} and in the previous paragraph to multiples of a single weighted shift acting on $\ell^p(\N)$, $1\leq p<+\infty$.

Let us fix $1\leq p<+\infty$. We recall that $\ell^p(\N)$ stands for the space of all sequences $x=(x_n)_{n\in \N}$ of scalars for which $\Vert x \Vert:=(\sum_{n\in \N}|x_n|^p)^{1/p} < +\infty$. Endowed with the norm $\Vert \cdot \Vert$, it is a Banach space. The unit sequence $(e_n)_{n\in \N}$ is a boundedly complete unconditional basis of $\ell^p(\N)$ and the subspace $X_0:=c_{00}(\N)=\text{span}(e_n:\, n\geq 0)$ is dense in $\ell^p(\N)$. A weighted shift $B_w$ is bounded on $\ell^p(\N)$ if and only if the sequence $w$ is bounded, \textit{i.e.} $\sup_{n \geq 1} \vert w_n \vert<+\infty$, in which case $\Vert B_w \Vert = \sup_{n \geq 1} \vert w_n \vert$.

Most of the important quantities introduced in Section \ref{subsec-multiple} can be explicitly computed when working with weighted shifts on  $\ell^p(\N)$. We keep the notations of Paragraph \ref{sufficientmultiple} except for the spectral radius $r(B_w)$ of a weighted shift $B_w$ that we will simply denote by $r_w$. We also set
\[
r_{p,w}:=\sup\{|\lambda|:\,\lambda \in \sigma_p(B_w)\},
\]
where $\sigma_p(B_w)$ denotes the point spectrum of $B_w$ (\textit{i.e.}, the eigenvalues of $B_w$). Then some calculations give:
\begin{itemize}
	\item $X_0 = \cup_{n\geq 1}\ker(B_w^n)$, hence $b_{B_w}(X_0,F_w)=+\infty$;
	\item $a_{B_w}(X_0,F_w)=r_{p,w}^{-1}=\limsup_{n \to \infty} \vert w_1\ldots w_n \vert^{-1/n}$, see e.g., \cite[Theorem 8, p. 70]{Shields};
	\item $r_w=\lim_{n \to \infty}\left(\sup_{k \geq 1} \vert w_k\ldots w_{k+n} \vert \right)^{1/n}$.
\end{itemize}
Let us also introduce the quantity:
\begin{itemize}\item $\lambda_{w}:=\limsup_{n \to \infty} \vert w_1\ldots w_n \vert^{1/n}$.
\end{itemize}
We thus have
\[
\Vert B_w \Vert ^{-1}\leq r_w^{-1}\leq \lambda_{w}^{-1}\leq r_{p,w}^{-1}.
\]

\medskip

On the one hand, if $w$ is a monotonic sequence (hence a convergent sequence to some real number $w_{\infty}$), then $r_w= \lambda_{w}= r_{p,w}=w_{\infty}$. Note that if $w$ is increasing, then these quantities are also equal to $\Vert B_w \Vert$. On the other hand, as shown by the next example, it is not difficult to provide with weights $w$ which allow to distinguish all or some of the quantities $\Vert B_w \Vert^{-1}$, $r_w^{-1}$, $\lambda_w^{-1}$ and $r_{p,w}^{-1}$.

\begin{ex} Let $a\leq b\leq c\leq d$ be four positive real numbers, and let us define, for any $n \geq 1$,
	\begin{equation*}
	w_n:=
	\begin{cases}
	a & { if }\ n\in \{1,\ldots,4\}\cup\{k2^{(k-1)^2} + 1,\ldots,2^{k^2}-1\}\\
	d & { if }\ n=2^{k^2}\\
	c & { if }\ n\in \{2^{k^2}+1,\ldots, 2^{k^2}+k+1\}\\
	b & { if }\ n\in \{2^{k^2}+k+2,\ldots, (k+1)2^{k^2}\}
	\end{cases}
	,\ k\geq 2.
	\end{equation*}
Then, one may check that
	\[
	\Vert B_w\Vert ^{-1}=1/d \leq r_w ^{-1}=1/c \leq \lambda_w^{-1} =1/b \leq r_{p,w}^{-1} =1/a.
	\]
\end{ex}

\medskip

We recall that by Bayart and Ruzsa's theorem \cite{BayRuz}, a weighted shift is frequently hypercyclic on $\ell^p(\N)$, $1\leq p< +\infty$, if and only if it satisfies the Frequent Hypercyclicity Criterion. Then, for any $0\leq \lambda < r_{p,w}^{-1}$, $\lambda B_w$ is not frequently hypercyclic. Together with Proposition \ref{optthmmult} and Corollary \ref{corothm29opt}, we thus have the following so far:

\begin{coro}\label{CorB_w}
	Let $B_w$ be a bounded weighted shift on $\ell^p(\N)$, $1\leq p < +\infty$, and let $\Lambda \subset (0,+\infty)$ be a non-empty set. Then
	\begin{enumerate}
		\item \label{eqrefrappel} the set $\bigcap_{\lambda \in \Lambda}FHC(\lambda B_w)$ is non-empty whenever $\Lambda$ is a countable relatively compact subset of $(r_{p,w}^{-1},+\infty)$;
		\item \label{eqrefrappel2} the set $\bigcap_{\lambda \in \Lambda}FHC(\lambda B_w)$ is empty whenever $\Lambda$ is unbounded, or $\Lambda$ has at least two elements and $r_w^{-1} \geq \inf(\Lambda)$.
	\end{enumerate}
In particular, if $r_{p,w}=r_w$ and $\Lambda$ has at least two elements, then the sufficient condition in \eqref{eqrefrappel} is also necessary.
\end{coro}

The next proposition is a slight improvement of \eqref{eqrefrappel2} in the previous corollary, and a partial answer to Question \ref{questopt} in the present context.

\begin{prop}\label{corooptmult}Let $B_w$ be a weighted shift acting on $\ell^p(\N)$ and $\Lambda$ a subset of $(0,+\infty)$ with at least two elements. If $\lambda_w^{-1} \geq \inf(\Lambda)$, then
	\[
	\bigcap_{\lambda \in \Lambda}FHC(\lambda B_w)=\emptyset.
	\]
\end{prop}

\begin{proof}It is very similar to that of Proposition \ref{optthmmult}, so we only give the outline in the case where $\Lambda$ is a sequence $(\lambda_k)_{k\in \N}$ decreasing to some $\lambda_{\infty} \leq \lambda_w^{-1}$, and where there exists $x=(x_n)_{n\in \N}\in \ell^p(\N)$ which is a hypercyclic vector for each $\lambda_kB_w$, $k\in \N$. As in the proof of Proposition \ref{optthmmult}, we introduce the sets
	\[
	\mathcal{N}_0:=\bigg\{n\in \N:\, \Vert \lambda_0^nB_w^n (x) \Vert < 1\bigg\}\quad 
	\text{and}
	\quad \mathcal{N}_k:=\bigg\{m\in \N:\, \Vert \lambda_k^mB_w^m(x)  - e_0\Vert < \frac{1}{2}\bigg\},\, k\geq 1.
	\]
Then we similarly define increasing sequences $(n_k)_{k\geq 1}\subset \mathcal{N}_0$ and $(m_k)_{k\geq 1}$, with $m_k \in \mathcal{N}_k$ for $k \geq 1$, and such that $\underline{d}(\mathcal{N}_0)\leq \limsup_{k \to \infty} n_k/m_k$. Thus for any $k\geq 1$,
	\[
	\lambda_0^{n_k} \vert w_{m_k-n_k+1}\ldots w_{m_k} \vert \vert x_{m_k}\vert < 1\quad {\rm and} \quad \lambda_k^{m_k} \vert w_1\ldots w_{m_k}\vert \vert x_{m_k}\vert > \frac{1}{2}.
	\]
It follows, for any $k \geq 1$,
	\[
	\frac{\lambda_0^{n_k}}{\lambda_k^{m_k}} < 2 \vert w_1 \ldots  w_{m_k-n_k} \vert.
	\]
In particular $m_k-n_k \to +\infty$ and for any $k \geq 1$,
	\[
	(\lambda_0/\lambda_k)^{n_k/m_k} < 2^{1/m_k} \lambda_k^{1-n_k/m_k} \vert w_1 \ldots w_{m_k-n_k}\vert^{1/m_k},
	\]
whence
	\[
	\underline{d}(\mathcal{N}_0)\leq \limsup_{k \to \infty}\frac{n_k}{m_k}\leq C(\limsup_{k \to \infty}\ln (\lambda_k) - \ln (\lambda_w^{-1}))\leq 0,
	\]
for some constant $C\geq 0$. Thus $x$ is not frequently hypercyclic for $\lambda_0 B_w$.
\end{proof}

\medskip

We then deduce the following.

\begin{coro}Let $B_w$ be a weighted shift on $\ell ^p(\N)$ and let $\Lambda$ be a subset of $(0,+\infty)$ with at least two elements. We assume that $\lambda_w=r_{p,w}$. Then
	\[
	\bigcap_{\lambda \in \Lambda}FHC(\lambda B_w)\neq \emptyset
	\]
if and only if $\Lambda$ is relatively compact in $(r_{p,w}^{-1},+\infty)$.
\end{coro}
	
The question whether the last corollary holds true for any weighted shift remains open. More precisely,

\begin{quest}Does the conclusion of Proposition \ref{corooptmult} hold true if $\lambda _w^{-1}$ is replaced by $r_{p,w}^{-1}$?
\end{quest}

\medskip

We conclude by applying the results of this paragraph in order to exhibit explicit frequently hypercyclic weighted shifts which share no frequently hypercyclic vector.

\begin{coro}\label{FHCmaisnoncommun}There exist two frequently hypercyclic weighted shifts on $\ell^p(\N)$, $1\leq p< +\infty$, with no common frequently hypercyclic vector.
\end{coro}

\begin{proof}Let $(w_n)_{n \geq 1} = \big((\frac{n+1}{n})^2\big)_{n \geq 1}$. Since $(w_n)_{n \geq 1}$ is decreasing to $1$, one has $r_w ^{-1}=r_{p,w}^{-1}=\lambda_w^{-1}=1$. Moreover, $B_w$ is frequently hypercyclic, since $\sum _{n\geq 1}(w_1\ldots w_n)^{-1}<\infty$. Thus, applying Proposition \ref{corooptmult} with $\Lambda =\{1,\lambda\}$, $\lambda >1$, we get $FHC(B_w)\cap FHC(\lambda B_w)=\emptyset$.
\end{proof}

\begin{rem} The proof of Corollary \ref{FHCmaisnoncommun} shows a bit more striking fact: for any monotonic weight $w$ converging to $w_{\infty}>0$, $w_{\infty}^{-1} B_w$ shares a common frequently hypercyclic vector with none of its multiple (different from itself).
\end{rem}

\subsection{Other examples}

In this paragraph, we apply our general common frequent hypercyclicity criterion (Theorem \ref{thmgeneral}) to classical frequently universal sequences of operators which are not weighted shifts.

Since almost all the classical examples of frequently hypercyclic operators satisfy the Frequent Hypercyclicity Criterion, the range of applications of Theorem \ref{thmmult} is quite large.

\begin{ex}[Differential operators on $H(\C)$] Let $D$ be the differentiation operator on $H(\C)$, $D(f)=f'$. Costakis and Mavroudis showed \cite{CosMav} that for any non-constant polynomial $P$, $P(D)$ satisfies Bayart and Matheron's criterion (Theorem \ref{BayMathcommon}) with $a_{P(D)}(X_0,S)=0$ and $b_{P(D)}(X_0,S)=+\infty$ for some dense subset $X_0$ of $H(\mathbb{C})$ and some right inverse $S$ of $P(D)$ on $X_0$. Thus, with the frequently hypercyclic version of the Le\'on-M\"uller Theorem and Theorem \ref{thmmult}, we can deduce that
\[
\bigcap_{\lambda \in \Lambda}FHC(\lambda P(D))\neq \emptyset,
\]
for any countable relatively compact non-empty subset $\Lambda$ of $\C^*$.
\end{ex}

We shall now focus on applications of Theorem \ref{thmgeneral} to families of operators which are not multiples of a single one. 

\begin{ex}[Adjoint of a multiplication operator on the Hardy space] We denote by $\D:=\{z\in \C:\, |z|<1\}$ the unit disc, by $H^{\infty}$ the space of bounded analytic functions in $\D$, and by $H^2$ the classical Hardy space,
\[
H^2:=\left\{f(z)=\sum _{k\geq 0}a_kz^k\in H(\D):\, \Vert f \Vert_2 := \bigg(\sum _{k\geq 0}|a_k|^2\bigg)^{1/2}<\infty\right\}.
\]
We recall that $H^2$ and $H^{\infty}$ are Banach spaces, endowed respectively with $\Vert \cdot \Vert _2$ and the $\sup$-norm $\Vert \cdot \Vert _{\infty}$. Let $\Phi \in H^{\infty}$ 
be such that $\Phi$ is not outer and $1/\Phi \in H^{\infty}$. We denote by $M_{\Phi}:H^2\to H^2$ the multiplication operator with symbol $\Phi$, defined by, $M_{\Phi}(f)=\Phi f$ for $f \in H^2$, and by $M_{\Phi}^*$ its adjoint. It is known \cite{bm} that $\lambda M_{\Phi}^*$ is frequently hypercyclic on $H^2$ for any $\lambda>\Vert 1/\Phi\Vert _{\infty}$ and that
\[
\bigcap_{\lambda >\Vert 1/\Phi\Vert _{\infty}}HC(\lambda M_{\Phi}^*)
\]
is a dense $G_{\delta}$-subset of $H^2$ \cite{GalPar}.

Now, let us write the inner-outer decomposition $\Phi=u\theta$, with $u$ outer and $\theta$ the non-constant inner part of $\Phi$. Let us define $X_0:=\cup_{n\geq 1}K_n$ with $K_n:=H^2\ominus \theta^nH^2$. Then $X_0$ is the generalized kernel of $M_{\Phi}^*$ and is dense in $H^2$. Moreover, if we define $S:=M_{1/u}^*M_{\theta}$, then $M_{\Phi}^*S={\rm Id}$ and $\Vert S \Vert = \Vert 1/\Phi\Vert _{\infty}$. We refer e.g., to the proof of \cite[Theorem 3.1]{GalPar} for the details concerning the previous claims. It is also known that $r(M_{\Phi}^*)=\Vert \Phi \Vert _{\infty}$. Thus we have
\[
\Vert \Phi \Vert _{\infty}^{-1} = r(M_{\Phi}^*)^{-1}\leq a_{M_{\Phi}^*}(X_0,S)\leq \Vert S \Vert = \Vert \Phi^{-1}\Vert _{\infty}
\]
and $b_{M_{\Phi}^*}(X_0,S)=+\infty$. Therefore, Theorem \ref{thmmult} directly implies that
\[
\bigcap_{\lambda \in \Lambda}FHC(\lambda M_{\Phi}^*)\neq \emptyset,
\]
whenever $\Lambda$ is a countable relatively compact non-empty subset of $(\Vert 1/\Phi \Vert _{\infty},+\infty)$.

\medskip

In fact, we can deduce from Corollary \ref{CFHC} the following more general result.
\begin{prop}Let $\{\Phi_{\lambda}:\,\lambda \in \Lambda\}$ be a countable family of bounded analytic functions in $\D$ with the same non-constant inner factor $\theta$. We assume that
\[
\sup\{\Vert 1/\Phi_{\lambda} \Vert _{\infty}: \lambda \in \Lambda\}<1\quad {\rm and} \quad \sup\{\Vert \Phi_{\lambda}/\Phi_{\mu} \Vert _{\infty}: \lambda,\mu \in \Lambda\}<\infty.
\]
Then
\[
\bigcap _{\lambda \in \Lambda}FHC(M^*_{\phi_{\lambda}})\neq \emptyset.
\]
\end{prop}

\begin{proof}We aim to apply Corollary \ref{CFHC}. By the comment after its statement, we need only check items (2)--(4) and the first part of (1). Since the functions $\Phi_{\lambda}$ share the same non-constant inner factor, the set $X_0:=\cup_{n\geq 1}K_n$ with $K_n:=H^2\ominus \theta^nH^2$ is the generalized kernel of each $M^*_{\phi_{\lambda}}$. More precisely, for any $f\in X_0$, there exists $n\geq 1$ such that $(M^*_{\phi_{\lambda}})^n(f)=0$ for any $\lambda \in \Lambda$. Let $u_{\lambda}$ denote the outer factor of $\Phi_{\lambda}$, $\lambda \in \Lambda$. As recalled above, setting $S_{\lambda}:=M^*_{1/u_\lambda}M_{\theta}$, we have $M^*_{\phi_{\lambda}}S_{\lambda}={\rm Id}$, $\lambda \in \Lambda$. So the first part of (1) and (4) of Corollary \ref{CFHC} are satisfied. Similarly, for any positive integers $n\leq m$ we have $(M^*_{\phi_{\lambda}})^mS_{\mu}^{m-n}=(M^*_{\phi_{\lambda}}S_{\mu})^{m-n}(M^*_{\phi_{\lambda}})^n$, $\lambda ,\mu \in \Lambda$, which gives (3) of Corollary \ref{CFHC}.
	
Let us prove (2) of Corollary \ref{CFHC}. We set
\[
a:=\sup\{\Vert 1/\Phi_{\lambda}^{-1} \Vert _{\infty}: \lambda \in \Lambda\}<1\quad {\rm and} \quad M:=\sup\{\Vert \Phi_{\lambda}/\Phi_{\mu} \Vert _{\infty}: \lambda,\mu \in \Lambda\}<\infty.
\]	
Let $\lambda \neq \mu \in \Lambda$ and $f\in X_0$.
By assumption, there exists $b\in (a,1)$ such that for any $m\in \N$ and $n\geq(c-1)m$,
\begin{align*}
\Vert (M^*_{\phi_{\lambda}})^m(S_{\mu}^{m+n}(f)) \Vert _2 & 
= \sup_{\Vert g \Vert _2=1}\left<(M^*_{\phi_{\lambda}})^m(S_{\mu}^{m+n}(f)),g\right>\\
& = \sup_{\Vert g \Vert _2=1}\left<f,\left(\frac{u _{\lambda}}{u_{\mu}}\right)^m \left(\frac{\bar{\theta}}{u_{\mu}}\right)^n g\right>\\
& \leq \Vert f \Vert _2 \left\Vert \frac{u _{\lambda}}{u_{\mu}}\right\Vert _{\infty}^m \left\Vert \frac{1}{u_{\mu}}\right\Vert _{\infty} ^n\\
& \leq \Vert f \Vert _2 \left(Mb^{(c-1)}\right)^{m}b^{n-(c-1)m}\left(\frac{a}{b}\right)^{n}.
\end{align*}
Since $b\in (a,1)$, (2) of Corollary \ref{CFHC} then follows by taking $c>1$ so that $Mb^{(c-1)}\leq 1$.
\end{proof}
\end{ex}

\section{Periodic points at the service of common frequent hypercyclicity}\label{sec-per}

Despite its apparent unpleasant formulation, the classical Frequent Hypercyclicity Criterion turns out to be very useful for checking that natural operators are frequently hypercyclic (and chaotic). We saw in the previous section that it fits well to formulating easy-to-use sufficient conditions for common frequent hypercyclicity. In \cite{GMM}, the authors provided a quite appealing new criterion for frequent hypercyclicity and chaos involving the periodic points of the operator \cite[Theorem 5.31]{GMM}. It is shown there that all the operators which satisfy the Frequent Hypercyclicity Criterion do also satisfy the assumptions of this new one. However, it quickly appears from its statement that it is not so simple to use when dealing with natural operators (e.g., weighted shifts). In fact, it is very well adapted to certain type of operators - the so called operators of $C$-type - which were introduced in \cite{MenetTransactions} and extensively developed \cite[Section 6]{GMM} in order to build several counter-examples.

In this section, we provide with a sufficient condition for common frequent hypercyclicity derived from \cite[Theorem 5.31]{GMM}. In the whole section, $X$ is a separable Banach space. We recall that a vector $x\in X$ is a periodic point for $T\in \LL(X)$ if there exists $p\in \N$ such $T^p(x)=x$. Let us denote by $\text{Per}(T)$ the set of all periodic points for $T$. For $x\in \text{Per}(T)$ we denote by $p_T(x)$ the period of $x$ for $T$ (\textit{i.e.}, the smallest positive integer $p$ such that $T^p(x)=x$).

\begin{thm}\label{thmFHCGMM}Let $X$ be a separable Banach space and $(T_s)_{s\geq 1}$ a countable family of bounded linear operators on $X$. We assume that there exist a dense linear subspace $X_0$ of $X$ with $T_s(X_0)\subset X_0$ and $X_0\subset \text{Per}(T_s)$ for any $s\geq 1$, and a constant $\alpha \in (0,1)$ such that the following property holds true: for every $s,q\geq 1$, every $\varepsilon>0$ and every $x,y\in X_0$,
there exist $z\in X_0$ and integers $n,d\geq 1$ such that, for every $1\leq t \leq q$,
	\begin{enumerate}
		\item $d$ is a multiple of $p_{T_{t}}(y)$ and of $p_{T_{t}}(z)$;\label{Cond1}
		\item $\Vert T_t^k(z)\Vert<\varepsilon$ for every $0\leq k\leq \alpha d$;\label{Cond2}
		\item $\Vert T_s^{n+k}(z)-T_s^k(x)\Vert < \varepsilon$ for every $0\leq k\leq \alpha d$.\label{Cond3}
	\end{enumerate}
Then there exists a common frequently hypercyclic vector for the family  $(T_s)_{s \geq 1}$.
\end{thm}

If the family $(T_s)_{s\geq 1}$ is reduced to a single operator, then Theorem \ref{thmFHCGMM} is exactly \cite[Theorem 5.31]{GMM}. Observe that Theorem \ref{thmFHCGMM} does not apply to families of multiples of a single operator since $\text{Per}(T)\cap \text{Per}(\lambda T)=\emptyset$ in general.

It is natural to wonder whether two operators satisfying \cite[Theorem 5.31]{GMM} have a common frequently hypercyclic vectors. Corollary \ref{FHCmaisnoncommun} tells us that this is not the case: there exists two multiples of the same weighted shift, which satisfy the classical Frequent Hypercyclicity Criterion, and so \cite[Theorem 5.31]{GMM} as well, and which do not have common frequently hypercyclic vectors. These two operators do not have the same periodic vectors. Now we do not know whether two operators satisfying \cite[Theorem 5.31]{GMM} with the same $X_0$ and $\alpha$ automatically share a frequently hypercyclic vectors. Note that applying Theorem \ref{thmFHCGMM} demand more than ``each $T_s$ satisfies the assumptions of \cite[Theorem 5.31]{GMM} with the same $X_0$ and $\alpha$''.

\begin{proof}[Proof of Theorem \ref{thmFHCGMM}]The proof is greatly inspired by that of \cite[Theorem 5.31]{GMM}. Let $(x_l)_{l\geq 1}$ be a sequence of vectors in $X_0$, dense in $X$, and let $(I_{p}(s))_{p,s\geq 1}$ be a partition of $\N\setminus\lbrace 0 \rbrace$ such that each set $I_p(s)$ is infinite and has bounded gaps. Let us denote by $r_{p}(s)$ the maximal size of a gap for $I_p(s)$. We set $I_p(s):=\{j_m(p,s):\,m\geq 1\}$, where $(j_m(p,s))_{m\geq 1}$ is increasing. Remark that, by definition, $j_{m+1}(p,s)-j_m(p,s) \leq r_p(s)$ for every $m\geq1$. We also let $(y_j)_{j\in\N}$ be given by $y_j=x_p$ if $j\in I_p(s)$. Now we use the assumptions of the theorem to build, by induction on $j \geq 1$, a sequence $(z_{j})_{j \geq 1}$ of vectors in $X_0$ and increasing sequences of positive integers $(d_{j})_{j\geq 1}$ and $(n_{j})_{j\geq 1}$ such that the following properties hold, for any $p, s \geq 1$ and $j\in I_p(s)$:
	\begin{enumerate}[(i)]
		\item \label{11} $d_{j}$ is a multiple of $p_{T_t}(\sum_{i=1}^{j-1}z_{i})$ and $ p_{T_{t}}(z_{j})$ for every $t\geq1$ so that there exists $q\geq1$ such that $I_q(t) \cap [1, j] \neq \emptyset$;
		\item \label{22}$\Vert T_{t}^k(z_{j}) \Vert < 2 ^{-j}$ for every $0\leq k \leq \alpha d_{j}$ and every $t\geq1$ so that there exists $q\geq1$ such that $I_q(t) \cap [1, j] \neq \emptyset$;
		\item \label{33}$\Vert T_s^{n_{j}+k}(z_{j}) - T_s^k(y_j-\sum_{i=1}^{j-1}z_{i})\Vert < 2^{-j}$ for every $0\leq k \leq \alpha d_{j}$;
		\item \label{44}$n_{j}$ is a multiple of $p_{T_s}(\sum_{i=1}^{j-1}z_{i})$ and $\alpha d_{j} < n_{j} \leq d_{j}$;
		\item \label{55} $\alpha d_{j} > 4 d_{j-1}$, where $d_0 := 0$.
	\end{enumerate}
Conditions \eqref{11} and \eqref{22} are made possible by the assumptions of the theorem and the fact that the set
\[
\bigcup_{q\geq 1}\bigcup_{1\leq i\leq j}\{t\geq 1:\,i\in I_q(t)\}
\]
is finite for any $j\geq 1$. The choice of $n_j$ is possible thanks to several elementary facts which are explained in the first lines of the proof of Theorem 5.31 in \cite{GMM}.

By \eqref{22}, the sum $z:= \sum_{i\geq 1} z_{i}$ defines a vector in $X$. Let us check that $z$ is frequently hypercyclic for every $T_s$, $s\geq 1$. To do so, we will show that every $x_p$, $p\geq 1$, can be approximated as close as desired by iterates $T_s^n(z)$ where the exponents $n$ form a set with positive lower density.

Let thus $p,s\geq 1$ be fixed. For notational simplicity, we will denote $j_m(p,s)$ by $j_m$, for any $m \geq 1$. Then, for every $m\geq1$ we define by induction on $j \geq 0$ a family of sets $(A_{m,j})_{0\leq j < j_{m+1}-j_m}$ as follows:
\[
A_{m,0}:=\bigg\{n_{j_m}+kd_{j_m}+k'p_{T_s}(x_p):\,0\leq k' \leq \frac{\alpha d_{j_m}}{p_{T_s}(x_p)},\,0\leq k\leq \frac{\alpha d_{j_m+1}}{d_{j_m}}-2\bigg\},
\]
and, for $1\leq j< j_{m+1}-j_m$,
\[
A_{m,j}:=\bigcup_{1\leq k\leq \frac{\alpha d_{j_m+j+1}}{d_{j_m+j}}-1}(A_{m,j-1}+kd_{j_m+j}).
\]
As in the proof of \cite[Theorem 5.31, Equation (16)]{GMM}, one easily checks by induction that $\max (A_{m,j})\leq \alpha d_{j_m+j+1}$, for any $0 \leq j < j_{m + 1} - j_m$ and $m \geq 1$. Moreover, by \cite[Fact 5.35]{GMM} (in fact exactly reproducing its proof), we have $\dlow(A)>0$ where
\[
A:= \bigcup _{m\geq1}\bigcup_{0\leq j < j_{m+1}-j_m}A_{m,j}.
\]
Note that $A$ depends on the fixed parameters $s$ and $p$. Thus to finish the proof of the theorem, we need only prove that for every $m\geq 1$ and every $0\leq j < j_{m+1}-j_m$, we have
\[
\Vert T_s^n (z) -x_p\Vert \leq 2^{-(j_m-1)},\quad n\in A_{m,j}.
\]
This shall be proven as in \cite[Fact 5.34]{GMM} up to some modifications. Let $m\geq 1$ and $0\leq j < j_{m+1}-j_m$, we first observe that for any $n\in A_{m,j}$ we have
\[
\Vert T_s^n (z) -x_p\Vert \leq \bigg\Vert T_s^n \bigg(\sum _{i=1}^{j_m+j}z_{i}\bigg) - x_p \bigg\Vert + \sum _{i>j_m+j} \bigg\Vert T_s^n (z_{i}) \bigg\Vert.
\]
Since $\max (A_{m,j})\leq \alpha d_{j_m+j+1}$, we have $n\leq \alpha d_{j_m+j+1}\leq \alpha d_{i}$ for every $i>j_m+j$ and $n \in A_{m, j}$, and it follows from \eqref{22} that
\[
\sum_{i>j_m+j} \Vert T_s^n (z_{i}) \Vert < \sum_{i>j_m+j} 2^{-i}\leq \frac{1}{2^{j_m+j}}.
\]
To conclude we now turn to proving that for every $m\geq 1$, $0\leq j < j_{m+1}-j_m$ and $n\in A_{m,j}$,
\begin{equation}\label{eqGMM}
\bigg\Vert T_s^n \bigg(\sum _{i=1}^{j_m+j}z_{i}\bigg) - x_p \bigg\Vert \leq \sum_{i=0}^j2^{-(j_m+i)}.
\end{equation}
Let $m \geq 1$. We proceed by induction on $0\leq j < j_{m+1}-j_m$. If $n\in A_{m,0}$ (\textit{i.e.}, $j=0$) then $n=n_{j_m}+kd_{j_m}+k'p_{T_s}(x_p)$ with $0\leq k\leq \frac{\alpha d_{j_m+1}}{d_{j_m}}-2$ and $0\leq k'\leq \frac{\alpha d_{j_m}}{p_{T_s}(x_p)}$, and by \eqref{11} and \eqref{44}
\begin{align*}
T_s^n \bigg(\sum _{i=1}^{j_m}z_{i}\bigg) - x_p & = T_s^{n_{j_m}+kd_{j_m}+k'p_{T_s}(x_p)}\bigg(\sum _{i=1}^{j_m}z_{i}\bigg)-x_p\\
& = T_s^{n_{j_m}+k'p_{T_s}(x_p)}(z_{j_m})-T_s^{k'p_{T_s}(x_p)}\bigg(x_p - \sum_{i=1}^{j_m-1}z_{i}\bigg).
\end{align*}
By \eqref{33} we get
\[
\bigg\Vert T_s^n \bigg(\sum _{i=1}^{j_m}z_{i}\bigg) - x_p \bigg\Vert \leq 2^{-j_m}.
\]

Assume now that \eqref{eqGMM} has been proven up to $j-1$ for some $1\leq j< j_{m+1}-j_m$. For $n\in A_{m,j}$, we write $n=kd_{j_m+j}+l$ with $l\in A_{m,j-1}$ and
\[
0\leq k \leq \frac{\alpha d_{j_m+j+1}}{d_{j_m+j}}-1.
\]
Then, by \eqref{11} we have
\begin{align*}T_s^n \bigg(\sum _{i=1}^{j_m+j}z_{i}\bigg) - x_p & = T_s^{kd_{j_m+j}+l}\bigg(\sum _{i=1}^{j_m+j}z_{i}\bigg)-x_p\\
& = T_s^{l}\bigg(\sum _{i=1}^{j_m+j-1}z_{i}\bigg)-x_p + T_s^{l}(z_{j_m+j}).
\end{align*}
Since $l\in A_{m,j-1}$, we deduce from the induction hypothesis and \eqref{22} that
\[
\bigg\Vert T_s^n \bigg(\sum _{i=1}^{j_m+j}z_{i}\bigg) - x_p \bigg\Vert \leq \sum_{i=0}^{j-1}2^{-(j_m+i)} + 2^{-(j_m+j)},
\]
and \eqref{eqGMM} as desired.
\end{proof}

\subsection*{Application to operators of $C$-type}

We will apply Theorem \ref{thmFHCGMM} to operators of $C$-type on $\ell^p(\N)$, $1\leq p<\infty$. First we shall recall their definition, following the formalism of \cite[Section 6]{GMM}. As usual, we denote by $(e_k)_{k\in \N}$ the unit sequence of $\ell^p(\N)$. An operator of $C$-type is associated with a data of four parameters $v,w,\varphi$ and $b$:

\begin{itemize}
	\item $v=(v_n)_{n\geq 1}$ is a sequence of non-zero complex numbers with $\sum_{n\geq 1}|v_n|<\infty$;
	\item $w=(w_n)_{n\geq 1}$ is a sequence of complex numbers such that
	\[
	0<\inf_{n\geq 1}|w_n|\leq \sup_{n\geq 1}|w_n|<\infty;
	\]
	\item $\varphi : \N \to \N$ is such that $\varphi(0)=0$, $\varphi(n)<n$ for every $n\geq 1$, and the set $\varphi^{-1}(\lbrace l \rbrace)$ is infinite for every $l\geq 0$;
	\item $b=(b_n)_{n\geq 0}$ is an increasing sequence of integers with $b_0=0$ and $b_{n+1}-b_n$ is a multiple of $2(b_{\varphi(n)+1}-b_{\varphi(n)})$ for every $n\geq 1$.
\end{itemize}

Now, for a data $v,w,\varphi$ and $b$ as above, the operator of $C$-type $T_{v,w,\varphi,b}$ is defined by
\[
T_{v,w,\varphi,b}(e_k)=\left\{\begin{array}{ll}
w_{k+1}e_{k+1} & \text{if } k\in[b_n,b_{n+1}-1),\,n\geq 0\\
v_ne_{b_{\varphi(n)}}-\left(\prod _{j=b_n+1}^{b_{n+1}-1}w_j\right)^{-1}e_{b_n} & \text{if } k=b_{n+1}-1,\, n\geq 1\\
-\left(\prod _{j=b_0+1}^{b_1-1}w_j\right)^{-1}e_0 & \text{if } k=b_1-1.
\end{array}\right.
\]
Here, by convention, an empty product is equal to $0$. From now on, we assume that the condition
\[
\inf_{n\geq 0}\prod_{b_n<j<b_{n+1}}|w_j|>0
\]
is satisfied. As shown by \cite[Fact 6.2]{GMM}, this assumption ensures that $T_{v,w,\varphi,b}$ is a bounded operator from $\ell^p(\N)$ into itself. It can also be checked that each element of $c_{00}$ is a periodic point for $T_{v,w,\varphi,b}$, more precisely 
\[T_{v,w,\varphi,b}^{2(b_{n+1}-b_n)}e_k=e_k\text{ if }k\in[b_n,b_{n+1}), n\geq0\]
see \cite[Fact 6.4]{GMM}.

In order to deal with frequent hypercyclicity, the authors of \cite{GMM} introduce a subclass of operators of $C$-type. As we are interested in common frequent hypercyclicity, we will work within this subclass. It consists in those operators of $C$-type for which the data $v,w,\varphi,b$ has the following special structure: for every $k\geq 1$,
\begin{itemize}
	\item $\varphi(n)=n-2^{k-1}$ for every $n\in[2^{k-1},2^k)$;
	\item there exists $\Delta ^{(k)} \in \N$ such that the size of the block $[b_n,b_{n+1})$, \textit{i.e.} the quantity $b_{n+1}-b_n$, is equal to $\Delta^{(k)}$ for every $n\in[2^{k-1},2^k)$;
	\item there exists $v^{(k)} \in \C\setminus \{0\}$ such that $v_n=v^{(k)}$ for every $n\in[2^{k-1},2^k)$;
	\item there exists a sequence $(w_i^{(k)})_{1\leq i< \Delta^{(k)}}$ such that $w_{b_n+i}=w_i^{(k)}$ for every $1\leq i< \Delta^{(k)}$ and every $n\in[2^{k-1},2^k)$.
\end{itemize}

An operator of $C$-type which satisfies the previous conditions is called an operator of $C_+$-type. The next result is a criterion for a countable family of operators of $C_+$-type to share a common frequently hypercyclic vector.

\begin{thm}\label{TH_C_+}
	Let $(T_{v(s),w(s),\varphi,b})_{s\geq 1}$ be a countable family of operators of $C_+$-type on $\ell^p(\N)$ where $b$ does not depend on $s$. We assume that there exists a constant $\alpha>0$ such that for every $s\geq 1$, every $C\geq1$ and every $k_0\geq1$, there exists an integer $k\geq k_0$ such that, for every $0\leq n\leq \alpha \Delta^{(k)}$,
	\begin{equation}\label{Ceq2}
	\vert v^{(k)}(s)\vert \prod_{i=n+1}^{\Delta^{(k)}-1}\vert w_{i}^{(k)}(s)\vert>C.
	\end{equation}
	If, for any $s,t\geq 1$, there exists a constant $K_{s,t}>0$ such that for any $r\geq \rho \geq 1$,
	\begin{equation}\label{Ceq3}
	\left\vert\frac{w_\rho(t)w_{\rho+1}(t)\ldots w_r(t)}{w_\rho(s)w_{\rho+1}(s)\ldots w_r(s)}\right\vert\leq K_{s,t},
	\end{equation}
	then $\bigcap_{s\geq 1}FHC(T_{v(s),w(s),\varphi,b})$ is non-empty.
\end{thm}

Note that since $b$ does not depend on $s$, by definition, the integers $\Delta^{(k)}$, $k\geq 1$, do not depend on $s$ either. It is plainly checked that condition \eqref{Ceq2} is equivalent to saying that each $T_s$ satisfies the assumption of \cite[Theorem 6.9]{GMM}. In particular, if $\{T_{v(s),w(s),\varphi,b}:\,s\in \N\}$ is reduced to a single operator (\textit{i.e.}, $v(s)$ and $w(s)$ do not depend on $s$), then the previous criterion is exactly \cite[Theorem 6.9]{GMM}.

For the proof of Theorem \ref{TH_C_+}, we recall \cite[Fact 6.8]{GMM} below.

\begin{fact}\label{Fact6.8}
	Let $T$ be an operator of $C_+$-type on $\ell^{p}(\N)$ and $k\geq1$. For any $l<2^{k-1}$ and $1\leq m\leq \Delta^{(k)}$, we have 
	\[
	T^{m}(e_{b_{2^{k-1}+l+1}-m})
	=v^{(k)}\left(\prod_{i=\Delta^{(k)}-m+1}^{\Delta^{(k)}-1}w_{i}^{(k)}\right)e_{b_l}-\left(\prod_{i=1}^{\Delta^{(k)}-m}w_{i}^{(k)}\right)^{-1}e_{b_{2^{k-1}+l}}.\]
\end{fact}

\begin{proof}[Proof of Theorem \ref{TH_C_+}]
Without loss of generality, we can assume that $0<\alpha <1$. It suffices to check that the assumptions of Theorem \ref{thmFHCGMM} are satisfied. Since every finitely supported vector is periodic for $C_+$-type operators, we define $X_0:=c_{00}$ and fix $x,y\in X_0$, $\varepsilon>0$ and $s,q\geq 1$. There exists $k_0\geq1$ such that
\[x=\sum_{l<2^{k_0}}\sum_{j=b_l}^{b_{l+1}-1}x_je_j.\]
By condition \eqref{Ceq2}, for any $C>0$, there exists $k\geq k_0$ such that
\[
\vert v^{(k)}(s)\vert \prod_{i=n+1}^{\Delta^{(k)}-1}\vert w_{i}^{(k)}(s)\vert>C,\quad 0\leq n \leq \alpha\Delta^{(k)}.
\]
Since $v(s)$ and $w(s)$ are bounded, upon choosing $C$ large enough, we may assume that $k$ is so large that the following holds true:
\begin{enumerate}[(a)]
	\item  $\Delta^{(k_0)}< \min((1-\frac{\alpha}{2})\Delta^{(k)},\frac{\alpha}{2} \Delta^{(k)}-1)$\label{firsteq2};
	\item $\Delta^{(k)}$ is a multiple of $p_{T_{t}}(y)$ for any $t\geq 1$\label{firsteq1}.
\end{enumerate}
Indeed, as noted after the definition of $C$-type operators, since the period of any vector in $X_0$ can be explicitely expressed in terms of the sequence $b$, \eqref{firsteq1} is satisfied whenever $y$ is supported in $[0,b_{2^{k-1}}[$. Let us now set $n:=\Delta^{(k)}-1$, $d:=2\Delta^{(k)}$ and
\[
z:=\sum_{l<2^{k_0}}\sum_{j=b_l}^{b_{l+1}-1}x_j\left(v^{(k)}(s)\prod_{i=j-b_l+2}^{\Delta^{(k)}-1}w_{i}^{(k)}(s)\right)^{-1}\left(\prod_{i=1}^{j-b_l}w_{b_l+i}(s)\right)^{-1}e_{b_{2^{k-1}+l+1}-n+j-b_l}.
\]
Like for \eqref{firsteq1} above, $d$ is a multiple of $p_{T_s}(z)$ for any $s\geq 1$. Thus condition \eqref{Cond1} of Theorem \ref{thmFHCGMM} is satisfied.

Let us now fix $0\leq m\leq \frac{\alpha d}{4}$ and $1\leq t\leq q$. We observe that for every $l<2^{k_0}$ and $b_l\leq j\leq b_{l+1}-1$, we have
\[
b_{2^{k-1}+l+1}-n+j-b_l+m\in [b_{2^{k-1}+l},b_{2^{k-1}+l+1}).
\]
Indeed, by definition $b_{2^{k-1}+l+1}-b_{2^{k-1}+l}=\Delta^{(k)}$ and by \eqref{firsteq2}, $-\Delta^{(k)}\leq -n+j-b_l+m < 0$. So for every $t\geq1$, we have
\[
T_{t}^{m}(e_{b_{2^{k-1}+l+1}-n+j-b_l})
=\left(\prod_{i=\Delta^{(k)}-n+j-b_l+1}^{\Delta^{(k)}-n+j-b_l+m}w_{i}^{(k)}(t)\right)e_{b_{2^{k-1}+l+1}-n+j-b_l+m},\]
hence the expression
\begin{multline}
T_{t}^{m}(z)=\sum_{l<2^{k_0}}\sum_{j=b_l}^{b_{l+1}-1}x_j\left(v^{(k)}(s)\prod_{i=j-b_l+m+2}^{\Delta^{(k)}-1}w_{i}^{(k)}(s)\right)^{-1}\left(\prod_{i=1}^{j-b_l}w_{b_l+i}(s)\right)^{-1}\\\left(\prod_{i=j-b_l+2}^{j-b_l+m+1}\frac{w_{i}^{(k)}(t)}{w_{i}^{(k)}(s)}\right)e_{b_{2^{k-1}+l+1}-n+j-b_l+m}.
\end{multline}
Using \eqref{firsteq2}, we know that $0\leq j-b_l+m+1\leq \alpha \Delta ^{(k)}$ which, by \eqref{Ceq2}, \eqref{Ceq3} and the definition of $C_+$-type operators, implies that for some constant $A>0$ (independent of $k$),
\[
\Vert T_{t}^{m}(z)\Vert \leq \Vert x\Vert C^{-1}\max_{1\leq t'\leq q}(K_{s,t'})A^{\Delta^{(k_0)}}.
\]
Up to choosing $C$ large enough, we get \eqref{Cond2} in Theorem \ref{thmFHCGMM}.

Let us now estimate the norm of $T_{s}^{n+m}(z)-T_{s}^{m}(x)$ for $0\leq m\leq \frac{\alpha d}{4}$. By Fact \ref{Fact6.8}, we obtain
\begin{multline*}
T_{s}^{n-(j-b_l)}(e_{b_{2^{k-1}+l+1}-n+j-b_l})
=v^{(k)}(s)\left(\prod_{i=\Delta^{(k)}-n+j-b_l+1}^{\Delta^{(k)}-1}w_{i}^{(k)}(s)\right)e_{b_l}\\
-\left(\prod_{i=1}^{\Delta^{(k)}-n+j-b_l}w_{i}^{(k)}(s)\right)^{-1}e_{b_{2^{k-1}+l}}.
\end{multline*}
Applying $T_{s}^{j-b_l}$ yields
\begin{multline*}
T_{s}^{n} (e_{b_{2^{k-1}+l+1}-n+j-b_l})
=\left(v^{(k)}(s)\prod_{i=j-b_l+2}^{\Delta^{(k)}-1}w_{i}^{(k)}(s)\right)\left(\prod_{i=1}^{j-b_l}w_{b_l+i}(s)\right)e_{j}\\-\left(w_{j-b_l+1}^{(k)}(s)\right)^{-1}e_{b_{2^{k-1}+l}+j-b_l}.
\end{multline*}
Moreover, since $m+j-b_l<\Delta^{(k)}$, we have
\[
T_{s}^{m} (e_{b_{2^{k-1}+l}+j-b_l})
=\left(\prod_{i=j-b_l+1}^{j-b_l+m}w_{i}^{(k)}(s)\right)e_{b_{2^{k-1}+l}+j-b_l+m},
\]
hence
\begin{multline*}
T_{s}^{n+m}(e_{b_{2^{k-1}+l+1}-n+j-b_l})
=\left(v^{(k)}(s)\prod_{i=j-b_l+2}^{\Delta^{(k)}-1}w_{i}^{(k)}(s)\right)\left(\prod_{i=1}^{j-b_l}w_{b_l+i}(s)\right)T_{s}^{m}(e_{j})\\
-\left(w_{j-b_l+1}^{(k)}(s)\right)^{-1}\left(\prod_{i=j-b_l+1}^{j-b_l+m}w_{i}^{(k)}(s)\right)e_{b_{2^{k-1}+l}+j-b_l+m}.
\end{multline*}
By definition of $z$, it follows that
\begin{multline*}
T_{s}^{n+m}(z)\\=T_{s}^{m}(x)-\sum_{l<2^{k_0}}\sum_{j=b_l}^{b_{l+1}-1}x_j\left(v^{(k)}(s)\prod_{i=j-b_l+m+1}^{\Delta^{(k)}-1}w_{i}^{(k)}(s)\right)^{-1}\left(\prod_{i=1}^{j-b_l}w_{b_l+i}(s)\right)^{-1}e_{b_{2^{k-1}+l}+j-b_l+m}.
\end{multline*}
By assumption, we thus get
\[
\Vert T_{s}^{n+m}(z)-T_{s}^{m}(x) \Vert
\leq \Vert x\Vert C^{-1}A^{\Delta^{(k_0)}}
\]
and condition \eqref{Cond3} of Theorem \ref{thmFHCGMM} with $\alpha'=\frac{\alpha}{4}$, as desired.
\end{proof}

\begin{rem}\label{Rembdd} It is clear from the proof that the conclusion of Theorem \ref{TH_C_+} remains true under the following weaker (but less nice) assumption: we assume that there exists a constant $0<\alpha<1$ such that for every integers $s,q\geq 1$, every $C\geq1$ and every $k_0\geq1$, there exists $k\geq k_0$ such that for every $0\leq n\leq \alpha \Delta^{(k)}$,
	\begin{equation}\label{Ceq2bis}
	\vert v^{(k)}(s)\vert \prod_{i=n+1}^{\Delta^{(k)}-1}\vert w_{i}^{(k)}(s)\vert>C;
	\end{equation}
	\begin{equation}\label{Ceq1}
	\displaystyle\sup_{\substack{1\leq t\leq q\\ 0\leq j< \Delta^{(k_0)}\\ 0\leq m\leq \alpha\Delta^{(k)}}} \left(\vert v^{(k)}(s)\vert\prod_{i=j+2}^{\Delta^{(k)}-1}\vert w_{i}^{(k)}(s)\vert\right)^{-1}\left(\prod_{i=j+2}^{j+m+1}\vert w_{i}^{(k)}(t)\vert\right)<\frac{1}{C}.
	\end{equation}

Moreover, \eqref{Ceq1} is satisfied whenever there exists a constant $0<\alpha<1$ such that for every integers $s,q\geq 1$, every $C\geq1$ and every $k_0\geq1$, there exists $k\geq k_0$ and $A>1$ such that
\begin{equation}\label{eqrem1} 
\sup_{i,t\geq 1}\max\left(\vert w_i(t)\vert,\frac{1}{\vert w_i(t)\vert}\right)\leq A;
\end{equation}
and 
\begin{equation}\label{eqrem2}
\displaystyle\sup_{\substack{1\leq t\leq q\\ 0\leq m\leq \alpha\Delta^{(k)}}} \left(\vert v^{(k)}(s)\vert\prod_{i=1}^{\Delta^{(k)}-1}\vert w_{i}^{(k)}(s)\vert\right)^{-1}\left(\prod_{i=1}^{m+1}\vert w_{i}^{(k)}(t)\vert\right)<\frac{1}{C}.
\end{equation}
Thus the conclusion of Theorem \ref{TH_C_+} remains true under the assumptions of Remark \ref{Rembdd} with \eqref{Ceq1} replaced by  \eqref{eqrem1} and \eqref{eqrem2}.
\end{rem}

It turns out that for a certain subclass of operators of $C_+$-type, for which \eqref{eqrem1} automatically holds, some rather simple condition for frequent hypercyclicity is given in \cite{GMM}. We shall now see that a similar condition for a family of operators in this subclass implies \eqref{eqrem2} and thus common frequent hypercyclicity.

\subsection*{Application to operators of $C_{+,1}$-type}

Operators of $C_{+,1}$-type are introduced in \cite[Section 6.5]{GMM} as those operators of $C_+$-type for which the parameters $v$ and $w$ satisfy the following extra condition: for every $k\geq 1$,
\[
v^{(k)}=2^{-\tau^{(k)}}\quad \text{and} \quad w_i^{(k)}=\left\{\begin{array}{ll}2 & \text{ if }1\leq i\leq \delta^{(k)}\\1 & \text{ if }\delta^{(k)}<i<\Delta^{(k)}\end{array}\right.,
\]
where $\tau:=(\tau^{(k)})_{k\geq 1}$ and $\delta:=(\delta^{(k)})_{k\geq 1}$ are two strictly increasing sequences of integers such that $\delta^{(k)}<\Delta^{(k)}$, $k\geq 1$. Within this class of operators of $C_{+,1}$-type, that we simply denote by $T_{\tau,\delta,\varphi,b}$, examples of frequently hypercyclic operators which are not ergodic were provided in \cite{GMM}.

\begin{thm}\label{TH_C_{+,1}}
Let $(T_{\tau(s),\delta(s),\varphi,b})_{s\geq 1}$ be a countable family of operators of $C_{+,1}$-type on $\ell^p(\N)$ where $b$ does not depend on $s$. If 
\[
\inf_{t\geq1}\limsup_{k\to\infty}\frac{\delta^{(k)}(t)-\tau^{(k)}(t)}{\Delta^{(k)}}>0,
\]
then $\bigcap_{s\geq 1}FHC(T_{\tau(s),\delta(s),\varphi,b})$ is non-empty.
\end{thm}

\begin{proof}
Observe that \eqref{eqrem1} in Remark \ref{Rembdd} trivially holds, thus it is enough to check \eqref{eqrem2} and \eqref{Ceq2bis}. To do so, we define
\[\alpha<\min\left(1,\frac{1}{2}\inf_{t\geq1}\limsup_{k\to\infty}\frac{\delta^{(k)}(t)-\tau^{(k)}(t)}{\Delta^{(k)}}\right).\]
Let $s,k_0\geq1$ and $C\geq 1$, and let us set $n=\Delta^{(k)}-1$. Since $\Delta^{(k)}\to\infty$ as $k\to \infty$, there exists $k\geq k_0$ such that
\[\frac{\delta^{(k)}(s)-\tau^{(k)}(s)}{\Delta^{(k)}}>2\alpha\quad \text{and}\quad \alpha\Delta^{(k)}>\log_{2}(C).\]
Then it follows from the definition of operators of $C_{+,1}$-type that \eqref{eqrem2} in Remark \ref{Rembdd} is equivalent to
\[
2^{\tau^{(k)}(s)-\delta^{(k)}(s)} \sup_{\substack{t\geq 1\\ 0\leq m\leq \alpha\Delta^{(k)}}} \prod_{i=1}^{m+1}\vert w_{i}^{(k)}(t)\vert<\frac{1}{C}.
\]
Now, we have 
\[
\sup_{\substack{t\geq 1\\ 0\leq m\leq \alpha\Delta^{(k)}}} \prod_{i=1}^{m+1}\vert w_{i}^{(k)}(t)\vert\leq 2^{\alpha\Delta^{(k)}}\leq 2^{\frac{1}{2}(\delta^{(k)}(s)-\tau^{(k)}(s))}.
\]
Hence,
\[
2^{\tau^{(k)}(s)-\delta^{(k)}(s)} \sup_{\substack{t\geq 1\\ 0\leq m\leq \alpha\Delta^{(k)}}} \prod_{i=1}^{m+1}\vert w_{i}^{(k)}(t)\vert\leq 2^{\frac{1}{2}(\tau^{(k)}(s)-\delta^{(k)}(s))}<2^{-\alpha\Delta^{(k)}}<\frac{1}{C}.
\]
It remains to check that  for every $0\leq n\leq \alpha \Delta^{(k)}$,
\[\vert v^{(k)}(s)\vert \prod_{i=n+1}^{\Delta^{(k)}-1}\vert w_{i}^{(k)}(s)\vert>C\]
which works the same as in the proof of \cite[Theorem 6.17]{GMM}.
\end{proof}

\begin{rem}
 When the family is reduced to a single operator, Theorem \ref{TH_C_{+,1}} is exactly \cite[Theorem 6.17]{GMM}.
\end{rem}

\section{Common frequent hypercyclicity with respect to densities}\label{Sec-alpha} 
We refer to \cite{Freedman} for the abstract definitions and the study of \emph{generalized lower and upper densities}. In particular it is proven there that to any sequence of non-negative real numbers $\alpha=(\alpha_k)_{k\geq 1}$ such that $\sum_{k\geq 1}\alpha_k=+\infty$, one can associate generalized lower and upper densities $\dlowg{\alpha}$ and $\dupg{\alpha}$ by the formulae
\[
\dlowg{\alpha}(E)=\liminf_{n \to \infty}\sum_{k\geq 1}\alpha_{n,k}\indicatrice{E}(k)\quad \text{and}\quad \dupg{\alpha}(E)=1-\dlowg{\alpha}(\N\setminus E),\quad E\subset \N,
\]
where $(\alpha_{n,k})_{n,k\geq 1}$ is the matrix given by 
\[
\alpha_{n,k}=\left\{\begin{array}{l}\alpha_k/(\sum_{j=1}^n\alpha_j)\hbox{ for }1\leq k\leq n,\\
0\hbox{ otherwise.}\end{array}\right.
\]
Then we also have $\dupg{\alpha}(E)=\limsup_{n \to \infty}\sum_{k=1}^{+\infty}\alpha_{n,k}\indicatrice{E}(k)$. By \cite[Lemma 2.7]{ErMo}, if we assume in addition that the sequence $(\alpha_n/(\sum_{j=1}^n\alpha_j))_{n\geq 1}$ converges to $0$, then for any set $E\subset \N$ enumerated as an increasing sequence $(n_k)_{k\geq 1}$, we have
\[
\dlowg{\alpha}(E)=\liminf_{k \to \infty}\frac{\sum_{j=1}^{k}\alpha_{n_j}}{\sum_{j=1}^{n_k}\alpha_j}.
\]

For $\alpha$ and $\beta$ two sequences as above, let us write $\alpha \lesssim \beta$ if there exists $k_0\in \N$ such that $(\alpha_k/\beta_k)_{k\geq k_0}$ is non-increasing. Then we have
\[
\dlowg{\beta}(E)\leq \dlowg{\alpha}(E)\leq \dupg{\alpha}(E) \leq \dupg{\beta}(E),\quad E\subset \N,
\]
whenever $\alpha \lesssim \beta$ (see \cite[Lemma 2.8]{ErMo}). Thus one can define scales of well-ordered densities with respect to the type of growth of the defining sequences. As we aim to study densities $\dlow_{\alpha}$ which are less than or equal to the natural one, it will be natural to assume that $\alpha$ is non-decreasing.

\medskip

From now on, a sequence $\alpha=(\alpha_k)_{k\geq 1}$ of non-negative numbers will be called \textit{completely admissible} if it satisfies the following two properties:

\begin{itemize}
	\item $\sum_{k\geq 1}\alpha_k=+\infty$;
	\item $\alpha$ is non-decreasing;
	\item $(\alpha_n/(\sum_{j=1}^n\alpha_j))_{n\geq 1}\to 0$ as $n\to \infty$.
\end{itemize}

A generalized density $\dlow_{\alpha}$ or $\dup_{\alpha}$ will be also called \textit{completely admissible} if it is associated to a completely admissible sequence $\alpha$.  Finally, the function $\varphi_{\alpha}:(0,+\infty)\to (0,+\infty)$ defined by $\varphi_{\alpha}(x)=\sum_{k\leq x}\alpha _k$, $n\geq 1$, will play an important role in the sequel.

\medskip

Several examples of generalized densities can be found in \cite{ErMo,ErMo2}. In this work, we will mainly be interested in four types of such sequences.
\begin{enumerate}\item For $0\leq \varepsilon\leq 1$, $\EE_{\varepsilon}:=(\exp(k^{\varepsilon}))_{k\geq 1}$. By a summation by parts, one can see that for $0< \varepsilon < 1$, $\varphi_{\EE_{\varepsilon}}(n) \sim \frac{n^{1-\varepsilon}}{\varepsilon}\exp(n^{\varepsilon})$ (where $u_k\sim v_k$ means $u_k/v_k\to 1$);
\item For $s\in \N\cup \{\infty\}$, $\DD_s:=(\exp(k/\log_{(s)}(k)))_{k\geq k_0}$ with $k_0$ large enough, where $\log_{(s)}=\log \circ \cdots \circ \log$, $\log$ appearing $s$ times, with the conventions $\log_{(0)}(x)=x$ and $\log_{(\infty)}(x)=1$ for any $x>0$. One can check that $\varphi_{\DD_s}(n) \sim \log_{(s)}(n)\exp (x/\log_{(s)}(n))$ for $s\in\N$ (see \cite[Remark 3.10]{ErMo2}) and $\varphi_{\DD_\infty}(n) \sim \frac{e}{e-1}\exp (n)$;
\item For all $l\geq 1$, let us consider the sequence $\mathcal{L}_l=(e^{\log(k)\log_{(l)}(k)})_{k\geq k_0}$, with $k_0$ large enough. A simple calculation leads to 
$\varphi_{\mathcal{L}_l}(n)\sim \frac{ne^{\log(n)\log_{(l)}(n)}}{\log_{(l)}(n)}$;
\item For $r\geq -1$ we shall also write $\PP_r:=(k^r)_{k\geq 1}$. Then $\varphi_{\PP_r}(k)\sim \frac{k^{r+1}}{r+1}$.
\end{enumerate}
Notice that for any $0\leq\varepsilon<1$, any $s\in\mathbb{N}$, any $l>1$ and any $r\geq 0$, 
the sequences $\EE_{\varepsilon}$,  
$\DD_s$, $\mathcal{L}_l$ and $\PP_r$ are completely admissible. Observe that the usual lower density $\dlow$ (associated to any constant sequence $(a,a,a,\ldots)$, $a>0$) corresponds to $\dlowg{\EE_0}$, $\dlowg{\DD_0}$ and $\dlowg{\PP_0}$. Later on, the sequence $\EE_1$ shall be simply denoted by $\EE$; note that $\dlowg{\EE}=\dlowg{\DD_{\infty}}$. For any $0< \delta \leq \varepsilon \leq 1$, any $s \leq t \in \N$, any $r\geq 0$ and any positive integer $l\leq l'$, we thus have
\[
\dlowg{\EE} \leq \dlowg{\DD_t} \leq \dlowg{\DD_s} \leq \dlowg{\EE_{\varepsilon}} \leq \dlowg{\EE_{\delta}} 
\leq \dlowg{\mathcal{L}_{l}}\leq  \dlowg{\mathcal{L}_{l'}}\leq\dlowg{\PP_r} \leq \dlow.
\]

\medskip

Let $X$ be a separable Fréchet space. As for frequently hypercyclic operators, we now say that a continuous linear operator on $T$ is $\alpha$-frequently hypercyclic if there exists $x\in X$ such that for any non-empty open set $U$ in $X$, $\dlowg{\alpha}(N(x,U,T))$ is positive. We denote by $FHC_{\alpha}(T)$ the set of all $\alpha$-frequently hypercyclic vectors for $T$. As proven in \cite{ErMo}, no operator can be $\EE$-frequently hypercyclic (and hence $\alpha$-frequently hypercyclic whenever $\EE \lesssim \alpha$).

\medskip

A first natural question arises:
\begin{quest}\label{q1sec3}Does common $\alpha$-frequent hypercyclicity exist for some $\alpha$?
\end{quest}

Let us recall that any operator satisfying the Frequent Universality Criterion is automatically $\alpha$-universal whenever $\alpha \lesssim \DD_s$ for some $s\geq 1$ \cite{ErMo2}. Since each of the criteria given in Section \ref{Sec2} are natural strengthenings of the Frequent Hypercyclicity Criterion, we could expect a positive answer to this question for any such $\alpha$. Moreover, it is easily seen that $FHC_{\PP_r}(T)=FHC(T)$ for any $r> -1$ (see \cite[Lemma 2.10]{ErMo}). So Question \ref{q1sec3} has a \textit{strong} positive answer for sequences with polynomial growth. In fact, the next proposition shows that for multiples of a single operator, the answer is either strongly positive, either strongly negative.

We will say that an increasing function $\varphi:(0,+\infty)\to (0,+\infty)$ satisfies the $\Delta_2$-condition if there exists a constant $K>0$ such that $\varphi(2x)\leq K\varphi(x)$ for any $x$ large enough.

\begin{thm}\label{propRomu}
Let $X$ be a separable Banach space, $T$ a bounded linear operator on $X$ and let $\alpha=(\alpha_k)_{k\geq 1}$ be a completely admissible sequence. Then
\begin{enumerate}\item \label{eqlabdens1}if $\varphi_{\alpha}$ satisfies the $\Delta_2$-condition, then $FHC(T)=FHC_{\alpha}(T)$;
	\item \label{eqlabdens2} if $\varphi_{\alpha}$ does not satisfy the $\Delta_2$-condition, then $HC(\lambda T)\cap FHC_{\alpha}(\mu T)=\emptyset$ for any $0<\lambda < \mu <+\infty$.
\end{enumerate}
\end{thm}

\begin{proof}To prove \eqref{eqlabdens1}, let us assume that $\varphi_{\alpha}$ satisfies the $\Delta_2$-condition. We need only check that for any $E\subset \N$, if $\dlow(E)$ is positive then $\dlow_{\alpha}(E)$ is also positive. Let us enumerate some $E\subset \N$ with $\dlow(E)>0$ by some increasing sequence $(n_k)_{k\geq 1}$. Thus, there exists $M>0$ with $k\leq n_k\leq Mk$ for any $k\geq 1$. Since $\alpha$ is completely admissible and $\varphi_{\alpha}$ satisfies the $\Delta_2$-condition, it is easy to check that there exists a constant $K$ depending on $M$, such that
\[
\sum_{j=1}^{n_k}\alpha_j\leq \sum_{j=1}^{Mk}\alpha_j\leq K\sum_{j=1}^{k}\alpha_j\leq K\sum_{j=1}^{k}\alpha_{n_j},
\]
for $k$ large enough, whence
\[
\dlowg{\alpha}(E)=\liminf_{k \to \infty}\frac{\sum_{j=1}^{k}\alpha_{n_j}}{\sum_{j=1}^{n_k}\alpha_j}\geq \frac{1}{K}.
\]

Let us now prove \eqref{eqlabdens2} and let us then fix $\alpha$ such that $\varphi_{\alpha}$ does not satisfy the $\Delta_2$-condition. It is not difficult to check that $\varphi_{\alpha}$ equivalently satisfies that for all $C>0$,
\begin{equation}\label{cond_surpoly}
\liminf_{k \to \infty}\left( \frac{\varphi_{\alpha}(k)}{\varphi_{\alpha}((1+C)k)}\right)=0,
\end{equation}
Let us also fix $0<\lambda < \mu <+\infty$. We shall assume that there exists $x\in HC(\lambda T)\cap HC(\mu T)$. Throughout the proof, $r>0$ is fixed. It is enough to prove the following:
\begin{enumerate}[(a)]
	\item \label{itemproof1}$\dupg{\alpha}(N(x,B(0,r),\lambda T))=1$;
	\item \label{itemproof2}$\dlowg{\alpha}(N(x,B(0,r),\mu T))=0$.
\end{enumerate}
We first prove \eqref{itemproof1}. By assumption, there exists an increasing sequence $(p_k)_{k\in \N}\subset \N$ such that $\Vert \mu^{p_k}T^{p_k}(x)\Vert < r$ for any $k\in \N$. Writing
\[
\lambda ^{p_k+i}T^{p_k+i}(x)=\lambda ^i T^i \left(\left(\frac{\lambda}{\mu}\right)^{p_k}\mu^{p_k}T^{p_k} (x)\right),\quad i\in \N,
\]
we easily check that $\Vert \lambda ^{p_k+i}T^{p_k+i}(x) \Vert < r$ whenever $(\lambda \Vert T \Vert )^i<(\mu /\lambda )^{p_k}$. Since by assumption $\lambda T$ is hypercyclic, we have $\lambda \Vert T \Vert  >1$. Thus there exists a constant $C>0$ (depending on $\lambda, \mu$ and $T$, but not on $k$) such that for any $i< Cp_k$, $\Vert \lambda ^{p_k+i}T^{p_k+i}(x) \Vert < r$. Therefore,
\[
\bigcup_{k\in \N}\big\{p_k,\ldots,\lfloor (1+C)p_k \rfloor\big\}\subset N(x,B(0,r),\lambda T).
\]
It follows thanks to (\ref{cond_surpoly}) that
\begin{equation*}\dupg{\alpha}(N(x,B(0,r),\lambda T))\geq 1-\liminf_{k \to \infty}\left( \frac{\varphi_{\alpha}(p_k)}{\varphi_{\alpha}((1+C)p_k)}\right)=1.
\end{equation*}

\eqref{itemproof2} is proved similarly. Since $x\in HC(\lambda T)$, there exists an increasing sequence $(p_k)_{k\in \N}\subset \N$ such that $\Vert \lambda ^{p_k} T^{p_k}(x) \Vert >r$. Writing $T^i\mu^{p_k-i}T^{p_k-i}= \mu ^{-i}(\mu /\lambda )^{p_k}\lambda^{p_k}T^{p_k}$, $1\leq i \leq p_k$, one can check that $\Vert \mu^{p_k-i}T^{p_k-i} (x) \Vert \geq r(\mu \Vert T \Vert )^{-i}(\mu /\lambda)^{p_k}$, $1\leq i \leq p_k$. Thus $\Vert \mu ^{p_k-i}T^{p_k-i}(x) \Vert > r$ whenever $(\mu \Vert T \Vert )^i > (\mu /\lambda)^{p_k}$. Since $\lambda \Vert T \Vert >1$, the last inequality is equivalent to $i\in \{\lfloor C p_k\rfloor+1,\ldots,p_k\}$ for some constant $0<C<1$ not depending on $k$. Therefore we get, using (\ref{cond_surpoly}) again,
\[
\dupg{\alpha}(N(x,X\setminus B(0,r),\mu T))  \geq 1 - \liminf_{k \to \infty} \left(\frac{\varphi_{\alpha}(\lfloor C p_k\rfloor+1)}{\varphi_{\alpha}(p_k)}\right)=1.
\]

\end{proof}

\begin{rem}\label{remsec3}{We shall mention that if the function $\varphi_{\alpha}$ is assumed to be convex and satisfies that there exists $\beta >1$ such that $\varphi_{\alpha}(\beta x) \geq 2 \beta \varphi_{\alpha}(x)$ for any $x$ large enough, then $\varphi_{\alpha}$ satisfies the $\Delta_2$-condition if and only if it has a growth controlled both from below and from above by some polynomials; \emph{i.e.}, if and only if there exist $r,r'\geq 1$ such that
\[
cx^{r'}\leq \varphi_{\alpha}(x)\leq Cx^r
\]
for some constants $c,C>0$ and any $x \in (1,+\infty)$. Next observe that on the scale of weighted densities given above the only class of sequences $\alpha$ for which $\varphi_{\alpha}$ satisfies the $\Delta_2$-condition 
is given by polynomials $\mathcal{P}_r$.
}
\end{rem}

We can illustrate the preceding proposition on our examples.

\begin{coro}
Let $X$ be a separable Banach space, $T$ a bounded linear operator on $X$ and $l\geq1$. Then for any $0<\lambda < \mu <+\infty$,
\[HC(\lambda T)\cap FHC_{\mathcal{L}_l}(\mu T)=\emptyset.\] 
\end{coro}

We shall mention that the non-existence of common frequently hypercyclic vectors in Theorem \ref{propRomu} concerns multiples $\lambda T$ and $\mu T$ of the same operator $T$ with $|\lambda|\neq |\mu|$. So we can still wonder whether common frequent hypercyclicity may exist for other kinds of families. Actually, this is the case if we consider families of unimodular multiples of a single operator, as the following extension of Le\'on-M\"uller's Theorem \cite[Theorem 6.28]{bm} shows.

\begin{thm}Let $X$ be a complex $F$-space, $T$ a $\alpha$-frequently hypercyclic operator on $X$ where $\alpha$ is a completely admissible sequence. Then $\lambda T$ is $\alpha$-frequently hypercyclic for any $\lambda\in\mathbb{C}$ with $\vert \lambda\vert =1$, and $FHC_{\alpha}(\lambda T)=FHC_{\alpha}(T)$.
\end{thm}

The proof goes along the same lines as that of \cite[Theorem 6.28]{bm}, replacing Lemma 6.29 by the following.
	\begin{lemme}
		Let $A\subset\N$ be a set of positive lower $\alpha$-density where $\alpha$ is completely admissible. Let also $I_1,\ldots,I_q\subset\N$ be such that $\cup_{j=1}^{q}I_j=\N$ and $n_1,\ldots,n_q\in\N$. Then $B:=\cup_{j=1}^{q}(n_j+A\cap I_j)$ has positive lower $\alpha$-density.  
	\end{lemme}
	
	\begin{proof}
		Let $N:=\max_{1\leq i\leq q}(n_i)$. On the one hand, for any $M\geq N$,
		\begin{align*}
		\frac{\sum_{k=1}^{M+N}\alpha_k \indicatrice{B}(k)}{\sum_{k=1}^{M+N}\alpha_k}&\geq\frac{1}{q}\frac{\sum_{j=1}^{q}\sum_{k=1}^{M+N}\alpha_k\indicatrice{n_j+A\cap I_j}(k)}{\sum_{k=1}^{M+N}\alpha_k}&\\
		&\geq \frac{1}{q}\frac{\sum_{j=1}^{q}\sum_{k=1}^{M}\alpha_k\indicatrice{A\cap I_j}(k)}{\sum_{k=1}^{M+N}\alpha_k}&\\
		&\geq \frac{1}{q}\frac{\sum_{k=1}^{M}\alpha_k\indicatrice{A}(k)}{\sum_{k=1}^{M+N}\alpha_k}\\
		&=\frac{1}{q}\frac{\sum_{k=1}^{M}\alpha_k\indicatrice{A}(k)}{\sum_{k=1}^{M}\alpha_k}\frac{\sum_{k=1}^{M}\alpha_k}{\sum_{k=1}^{M+N}\alpha_k}.
		\end{align*}
		On the other hand,
		\[\frac{\sum_{k=1}^{M}\alpha_k}{\sum_{k=1}^{M+N}\alpha_k}=1-\frac{\sum_{k=M+1}^{M+N}\alpha_{k}}{\sum_{k=1}^{M+N}\alpha_k}\geq 1-\left(\sum_{j=M+1}^{M+N}\frac{\alpha_{j}}{\sum_{k=1}^{j}\alpha_k}\right)\longrightarrow 1,\quad \text{as }M\to+\infty.\]
		Hence,
		\[\underline{d}_{\alpha}(B)=\liminf_{M \to \infty}\frac{\sum_{k=1}^{M+N}\alpha_k \indicatrice{B}(k)}{\sum_{k=1}^{M+N}\alpha_k}\geq\liminf_{M \to \infty} \frac{1}{q}\frac{\sum_{k=1}^{M}\alpha_k\indicatrice{A}(k)}{\sum_{k=1}^{M}\alpha_k}=\frac{1}{q}\ \underline{d}_{\alpha}(A)>0.\]
	\end{proof}

\medskip

To conclude, we come back to the main result of \cite{ErMo2}: any operator satisfying the Frequent Hypercyclicity Criterion is automatically $\alpha$-frequently hypercyclic if $\alpha \lesssim \DD_s$ for some $s\geq 1$. The following question naturally arises:

\begin{quest}\label{q2Sec3}For any operator $T\in \LL(X)$ satisfying the Frequent Hypercyclicity Criterion, does there exist a vector $x\in X$ which is $\alpha$-frequently hypercyclic for $T$ and for any $\alpha \lesssim \DD_s$ for some $s\geq1$?
\end{quest}

In \cite{ErMo2} the notion of $\alpha$-frequent universality is introduced as a natural extension of $\alpha$-frequent hypercyclicity for sequences of operators. In fact, the previously mentioned result of \cite{ErMo2} is proved for the notion of $\alpha$-frequent universality. Let us denote by $\mathcal{FU}_{\alpha}(\mathcal{T})$ the set of $\alpha$-frequently universal vectors for $\mathcal{T}\subset \LL(X)$. We give a positive answer to the previous question for that more general notion.

\begin{prop}\label{prop_common_dens}We denote by $\DD$ the set of all completely admissible sequences $\alpha=(\alpha_k)_{k\geq 1}$ such that $\alpha \lesssim \DD_s$ for some $s \in \N$. If $\mathcal{T} \subset \LL(X)$ satisfies the Frequent Universality Criterion, then
\[
\bigcap_{\alpha \in \DD}\mathcal{F}\mathcal{U}_{\alpha}(\mathcal{T})\neq \emptyset.
\]
\end{prop} 

\begin{proof}It is enough to prove that $\bigcap_{s\geq1}\mathcal{FU}_{\DD_s}(\mathcal{T})$ is non-empty. The proof is based on the calculations led in \cite[Section 3]{ErMo2}. Let us consider the function 
$f:\mathbb{N}\setminus\{0\}\rightarrow\mathbb{N}$ defined by $f(j)=m$ for all $j\in \{a_m,\dots,a_{m+1}-1\}$ with 
\[
a_1=1 \text{ and }a_m=2^{2^{\iddots^{2^{2^m}}}}\hbox{ where }2\hbox{ appears }m\hbox{ times for }m\geq2.
\]
Then we define the sequence $(n_k(f))_{k\geq 1}$ as follows:
\[
n_1(f)=2\quad \text{and}\quad n_k(f)=2\sum_{i=1}^{k-1}f(\delta_i)+f(\delta_k),\,k\geq 2,
\]
where $\delta_j$ is the index of the first zero in the dyadic representation of $j$ (e.g., if $k=11=1.2^0+1.2^1+0.2^2+1.2^3$, then $\delta_k=3$). Lemma 3.8 of \cite{ErMo2} (where one may check that $l_N=f(\lfloor \log(N)/\log(2)\rfloor)$) ensures that for all 
$s\geq 1$ there exist $C_1,C_2,C_3>0$ such that for all integers $k$ large enough,
$$C_1k-C_2\log_{(s)}(k)\leq n_k(f)\leq C_1k +C_3\log_{(s)}(k).$$

With such an asymptotic behavior, a similar calculation as that of \cite[Lemma 4.10]{ErMo} allows us to conclude that $\underline{d}_{\mathcal{D}_s}(\{n_k(f):\ k\in\N\})>0$ for all $s\geq 1$.
Moreover, the fact that $\underline{d}_{\DD_s}(\{n_k(f):\ k\in\N\})>0$ for every $s\geq1$ and that the set $\{k\in\N:\ \delta_k=p\}$ is an arithmetic progression for every $p\geq1$, imply $\underline{d}_{\DD_s}(\{n_k(f):\ \delta_k=p\})>0$ for every $p,s\geq1$.
Notice also that, as explained in \cite[Section~2]{ErMo2}, the proof of the classical Frequent Universality Criterion relies on a lemma ensuring the existence of countably many sets as above \cite[Lemma~9.8]{gp} satisfying $\underline{d}(\{n_k(f):\ \delta_k=p\})>0$ for every $p\geq1$. Hence, we can now mimic the proof of the classical Frequent Universality Criterion using the sets $\{n_k(f):\ \delta_k=p\}$, $p\geq 1$, to construct a vector which is $\DD_s$-frequently universal for $\mathcal{T}$ and for every $s\geq1$.
\end{proof}

\section{Remark: an ergodic approach?}\label{Sec-ergodic}

We shall conclude this paper by a word on the possible approach to common frequent hypercyclicity by ergodic theory. The most natural way to conceive frequent hypercyclicity is probably through Birkhoff's ergodic theorem: if $T$ is a bounded linear operator on some separable Banach space $X$, which is an ergodic measure-preserving transformation with respect to some measure $m$ with full support, then
\[
\lim_{N \to \infty}\frac{\text{card}(N(x,U,T)\cap[0,N])}{N+1}= m(U)>0,
\]
for any non-empty open set $U$ of $X$. In particular $T$ is frequently hypercyclic. We recall that $T$ is measure-preserving for $m$ if for any measurable set $A\subset X$, we have $m(A)=m(T^{-1}(A))$, and that $T$ is ergodic with respect to $m$ if for every measurable subsets $A$ and $B$ of $X$, with $m(A),m(B)>0$, there exists an integer $n$ such that $m(T^{-n}(A)\cap B)>0$ (see \cite{walters} for instance).

Now, if $T_1$ and $T_2$ are two ergodic measure-preserving transformations of $X$ for respectively two measures $m_1$ and $m_2$ which are absolutely continuous with respect to each other, then $T_1$ and $T_2$ automatically share a common frequently hypercyclic vector. Theorem 3.22 in \cite{Baygrifrequentlyhcop} gives a sufficient condition for an operator $T$ on $X$ to be ergodic and measure-preserving for some Gaussian measure $m_T$. In general, for two operators $T_1$ and $T_2$ satisfying this condition, $m_{T_1}$ and $m_{T_2}$ are not absolutely continuous with respect to each other and one cannot conclude whether they share a common frequently hypercyclic vector or not.

But the opposite situation can also occur: for example, let $B_{w_1}$ and $B_{w_2}$ be two weighted shifts on $\ell^2(\N)$ such that the supremum $r_{p,w_i}$ of their point spectrum is greater than $1$, $i=1,2$ (see Paragraph \ref{CFHCweightedshifts} where $r_{p,w_i}$ appears). Then by \cite[Theorem 3.22]{Baygrifrequentlyhcop}, $B_{w_1}$ and $B_{w_2}$ are ergodic and measure-preserving with respect to some Gaussian measures $m_{w_1}$ and $m_{w_2}$. For $i=1,2$, let us define $w_{i,n}=\prod_{k=0}^{n}w_i(k)$.  Now, as explained in \cite[Pages 5111-5112]{Baygrifrequentlyhcop}, $m_{w_1}$ and $m_{w_2}$ are absolutely continuous with respect to each other if and only if the sequence $(1-\sqrt{w_{1,n}/w_{2,n}})_{n\in \N}$ is in $\ell^2(\N)$. This condition is much stronger than the condition derived from the proof of Theorem \ref{weightedshiftscounta} (see Remark \ref{Rem_Th2.18}) which ensures the existence of common frequently hypercyclic vectors for more general families of weighted shifts. Note also that an ergodic approach has not permitted so far to obtain common frequent hypercyclicity for general multiples of a single operator. Yet, the fact that $(1-\sqrt{w_{1,n}/w_{2,n}})_{n\in \N}$ is in $\ell^2(\N)$ implies that $m_{w_1}(FHC(B_{w_1})\cap FHC(B_{w_2}))=1$, while our results give no quantitative information on the size of the set of common frequently hypercyclic vectors.

It would be of interest to investigate further the problem of common frequent hypercyclicity from the point of view of ergodicity.

\end{document}